\documentclass[12pt]{article}
\usepackage[numbers,sort&compress]{natbib}
\usepackage{color}
\usepackage[colorlinks=true, citecolor=black, linkcolor=black, urlcolor=blue,anchorcolor=blue,citecolor=black, filecolor=blue, menucolor=blue, pagecolor=blue,pdfcreator=Me, pdfproducer=Me]{hyperref}
\usepackage{enumitem}
\usepackage{microtype}
\usepackage[margin=30mm]{geometry}
\usepackage{amsmath,amsthm,amsfonts,amssymb}
\usepackage{calc,graphicx,float}

\usepackage[noabbrev,capitalise]{cleveref}
\crefname{lem}{Lemma}{Lemmas}
\crefname{thm}{Theorem}{Theorems}
\crefname{prop}{Proposition}{Propositions}

\theoremstyle{plain}
\newtheorem{thm}{Theorem}
\newtheorem{lem}[thm]{Lemma}
\newtheorem{cor}[thm]{Corollary}
\newtheorem{prop}[thm]{Proposition}

\newtheorem{obs}[thm]{Observation}
\newtheorem{op}{Open Problem}

\theoremstyle{definition}

\renewcommand{\baselinestretch}{1.15}
\usepackage[hang]{footmisc} 

\renewcommand{\thefootnote}{\fnsymbol{footnote}}	

\newcommand{\doi}[1]{doi:\,\href{http://dx.doi.org/#1}{#1}}

\newcommand\DateFootnote{
\begingroup
\renewcommand\thefootnote{}
\footnote{\today}
\setcounter{footnote}{0}
\vspace*{-3ex}
\endgroup}

\makeatletter
\renewcommand\section{\@startsection {section}{1}{\z@}%
                                   {-3ex \@plus -1ex \@minus -.2ex}%
                                   {2ex \@plus.2ex}%
                                   {\normalfont\large\bfseries}}
\renewcommand\subsection{\@startsection{subsection}{2}{\z@}%
                                     {-2.5ex\@plus -1ex \@minus -.2ex}%
                                     {1.5ex \@plus .2ex}%
                                     {\normalfont\normalsize\bfseries}}
\renewcommand\subsubsection{\@startsection{subsubsection}{3}{\z@}%
                                     {-2ex\@plus -1ex \@minus -.2ex}%
                                     {1ex \@plus .2ex}%
                                     {\normalfont\normalsize\bfseries}}
 \renewcommand\paragraph{\@startsection{paragraph}{4}{\z@}%
                                    {1.5ex \@plus.5ex \@minus.2ex}%
                                    {-1em}%
                                    {\normalfont\normalsize\bfseries}}
\renewcommand\subparagraph{\@startsection{subparagraph}{5}{\parindent}%
                                       {1.5ex \@plus.5ex \@minus .2ex}%
                                       {-1em}%
                                      {\normalfont\normalsize\bfseries}}
\makeatother

\renewcommand{\thefootnote}{\fnsymbol{footnote}}	
\newcommand{\N}{\mathbb{N}}

\newcommand{\ceil}[1]{\lceil{#1}\rceil}
\newcommand{\floor}[1]{\lfloor{#1}\rfloor}
\newcommand{\half}{\ensuremath{\protect\tfrac{1}{2}}}
\renewcommand{\geq}{\geqslant}
\renewcommand{\leq}{\leqslant}

\newcommand{\preproof}{\vspace*{-3ex}}

\DeclareMathOperator{\tw}{tw}
\DeclareMathOperator{\ttw}{2-tw}
\DeclareMathOperator{\pw}{pw}
\DeclareMathOperator{\bw}{bw}
\DeclareMathOperator{\cw}{cw}
\DeclareMathOperator{\crossnum}{cr}

\newcommand{\iters}{2k^2}
\renewcommand{\iters}{(k-1)^2}

\begin{document}
{\Large\bfseries\boldmath\scshape  Orthogonal Tree Decompositions of Graphs}

\DateFootnote

Vida Dujmovi{\'c}\,\footnotemark[1]
\quad Gwena\"el Joret\,\footnotemark[2]
\quad Pat Morin\,\footnotemark[3] 
\quad Sergey Norin\,\footnotemark[4] 
\quad David~R.~Wood\,\footnotemark[5] 

\footnotetext[1]{School of Computer Science and Electrical Engineering, University of Ottawa, 
Ottawa, Canada (\texttt{vida.dujmovic@uottawa.ca}). Supported by NSERC.}

\footnotetext[2]{D\'epartement d'Informatique, Universit\'e Libre de Bruxelles, Belgium (\texttt{gjoret@ulb.ac.be}). Supported by an Action de Recherches Concert\'ees grant from the Wallonia-Brussels Federation of Belgium}

\footnotetext[3]{School of  Computer Science, Carleton University,  Ottawa, Canada (\texttt{morin@scs.carleton.ca}). Research  supported by NSERC.}

\footnotetext[4]{Department of Mathematics and Statistics, McGill University, Montr\'eal, Canada (\texttt{snorin@math.mcgill.ca}). Supported by NSERC grant 418520.}

\footnotetext[5]{School of Mathematical Sciences, Monash University, Melbourne, Australia (\texttt{david.wood@monash.edu}). Supported by the Australian Research Council.}

\emph{Abstract.} 
This paper studies graphs that have two tree decompositions with the property that every bag from the first decomposition has a bounded-size intersection with every bag from the second decomposition. We show that every graph in each of the following classes has a tree decomposition and a linear-sized path decomposition with bounded intersections: (1) every proper minor-closed class, (2) string graphs with a linear number of crossings in a fixed surface, (3) graphs with linear crossing number in a fixed surface. Here `linear size' means that the total size of the bags in the path decomposition is $O(n)$ for $n$-vertex graphs. We then show that every $n$-vertex graph that has a tree decomposition and a linear-sized path decomposition with bounded intersections has $O(\sqrt{n})$ treewidth. As a corollary, we conclude a new lower bound on the crossing number of a graph in terms of its treewidth. Finally, we consider graph classes that have two path decompositions with bounded intersections. Trees and outerplanar graphs have this property. But for the next most simple class, series parallel graphs,  we show that no such result holds.

\renewcommand{\thefootnote}{\arabic{footnote}}

\newpage
\section{Introduction}
\label{Introduction}


A \emph{tree decomposition} represents the vertices of a graph as subtrees of a tree, so that the subtrees corresponding to adjacent vertices intersect. The \emph{treewidth} of a graph $G$ is the minimum taken over all tree decompositions of $G$, of the maximum number of pairwise intersecting subtrees minus 1. Treewidth measures how similar a given graph is to a tree. It is a key measure of the complexity of a graph and is of fundamental importance in algorithmic graph theory and structural graph theory. For example, treewidth is a key parameter in Robertson--Seymour graph minor theory \citep{RS-GraphMinors}, and many NP-complete problems are solvable in polynomial time on graphs of bounded treewidth \citep{ParameterizedAlgorithms}. 

The main idea in this paper is to consider two tree decompositions of a graph, and then measure the sizes of the intersection of bags from the first decomposition with bags from the second decomposition. Intuitively, one can think of the bags from the first decomposition as being horizontal, and the bags from the second  decomposition as being vertical, so that the two tree decompositions are `orthogonal' to each other. We are interested in which graphs have two tree decompositions such that every bag from the first decomposition has a bounded-size intersection with every bag from the second decomposition. This idea is implicit in recent work on layered tree decompositions (see \cref{LayeredTreewidth}), and was made explicit in the recent survey by \citet{Norine15}. 

Grid graphs illustrate this idea well; see \cref{Grid}. Say $G$ is the $n\times n$ planar grid graph. The sequence of consecutive pairs of columns determines a tree decomposition, in fact, a path decomposition with bags of size $2n$. Similarly, the sequence of consecutive pairs of rows determines a path decomposition with bags of size $2n$. Observe that the intersection of a bag from the first decomposition with a bag from the second decomposition has size 4. It is well known \citep{HW17} that $G$ has treewidth $n$, which is unbounded. But as we have shown, $G$ has two tree decompositions with bounded intersections. This paper shows that many interesting graph classes with unbounded treewidth have two tree decompositions with bounded intersections (and with other useful properties too).

\begin{figure}
\centering
\includegraphics{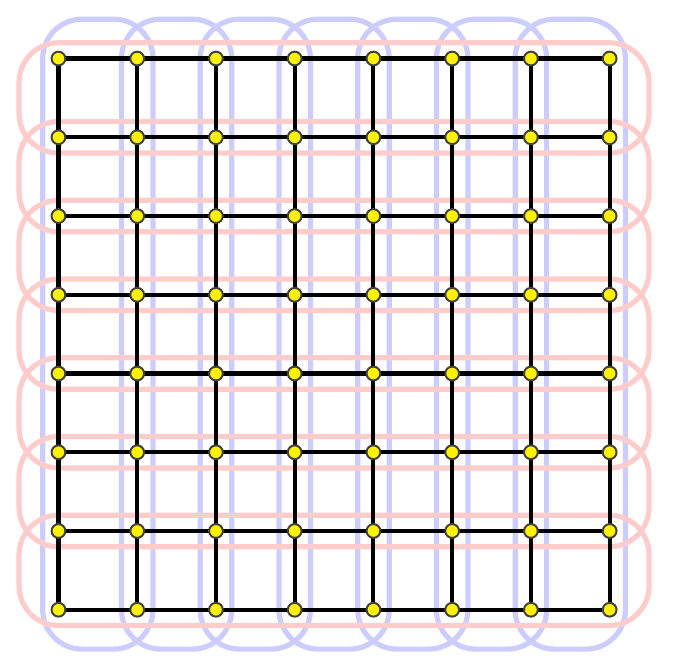}
\caption{ \label{Grid} Two 4-orthogonal path decompositions of the grid graph.}
\end{figure}

Before continuing, we formalise these ideas. A \emph{tree decomposition} of  a graph $G$ is given by a tree $T$ whose nodes index a collection $(T_x\subseteq V(G):x\in V(T))$ of sets of vertices in $G$ called  \emph{bags}, such that (1) for every edge $vw$ of $G$, some bag $T_x$ contains both $v$ and $w$, and (2) for every vertex $v$ of $G$, the set $\{x\in V(T):v\in T_x\}$ induces a non-empty (connected) subtree of $T$.
For brevity, we say that $T$ is a tree decomposition (with the bags $T_x$ implicit). 
The \emph{width} of a tree decomposition $T$ is $\max\{|T_x|-1:x\in V(T)\}$, and the \emph{treewidth} $\tw(G)$ of a graph $G$ is the minimum width of the tree decompositions of $G$. A \emph{path decomposition} is a tree decomposition in which the underlying tree is a path. We describe a path decomposition simply by the corresponding sequence of bags. The \emph{pathwidth} $\pw(G)$ of a graph $G$ is the minimum width of the path decompositions of $G$. Two tree decompositions $A$ and $B$ of a graph $G$ are \emph{$c$-orthogonal} if $|A_x\cap B_y|\leq c$ for all $x\in V(A)$ and $y\in V(B)$. 

It turns out that not only the size of bag intersections is important when considering orthogonal tree decompositions. A key parameter is the total size of the bags in a tree decomposition $T$, which we call the \emph{magnitude}, formally defined to be $\sum_{x\in V(T)}|T_x|$. For example, consider the complete bipartite graph $K_{n,n}$. Say $V=\{v_1,\dots,v_n\}$ and $W=\{w_1,\dots,w_n\}$ are the two colour classes. Then 
\begin{align*}
P & =(V\cup\{w_1\}, V\cup\{w_2\},\dots,V\cup\{w_n\})\text{ and }\\
Q& =(W\cup\{v_1\}, W\cup\{v_2\},\dots,W\cup\{v_n\})
\end{align*} 
are path decompositions of $K_{n,n}$, such that the intersection of each bag of $P$ with each bag of $Q$ has exactly two vertices. However, both $P$ and $Q$ have magnitude $n(n+1)$. 
On the other hand, we prove in \cref{Extremal} that if two tree decompositions of a graph $G$ have bounded intersections and one has linear magnitude, then $G$ has a linear number of edges. 
Here `linear' means $O(n)$ for $n$-vertex graphs.

Our main results show that every graph in each of the following classes has a tree decomposition and a linear-magnitude path decomposition with bounded intersections:
\begin{itemize}[itemsep=0ex,topsep=0ex]
\item every proper minor-closed class  (\cref{MinorClosed}),
\item string graphs with a linear number of crossings in a fixed surface (\cref{StringGraphs}),
\item graphs with linear crossing number in a fixed surface (\cref{CrossingNumber}),
\end{itemize}
The latter two examples highlight that orthogonal decompositions are of interest well beyond the world of minor-closed classes.  We also show that every graph that has a tree decomposition and a linear-magnitude path decomposition with bounded intersections has $O(\sqrt{n})$ treewidth. This result is immediately applicable to each of the above three classes. As a corollary, we conclude a new lower bound on the crossing number of a graph in terms of its treewidth (\cref{CrossingNumber}).

Treewidth is intrinsically related to graph separators. A set $S$ of vertices in a graph $G$ is a \emph{separator} of $G$ if each component of $G-S$ has at most $\half |V(G)|$ vertices. Graphs with small treewidth have small separators, as shown by the following result of  \citet{RS-GraphMinorsII-JAlg86}:

\begin{lem}[\citep{RS-GraphMinorsII-JAlg86}] 
\label{RS}
Every graph $G$ has a separator of size at most $\tw(G)+1$. 
\end{lem}

Our treewidth bounds and \cref{RS} give $O(\sqrt{n})$ separator results for each of the above three classes. Also note that a converse to \cref{RS}  holds: graphs in which every subgraph has a small separator have small treewidth \citep{DN14,Reed97}. 

The paper then considers graph classes that have two path decompositions with bounded intersections. Trees and outerplanar graphs have this property. But for the next most simple class, series parallel graphs,  we show that no such result holds (\cref{TwoPaths}). The paper concludes by discussing connections between orthogonal tree decompositions and boxicity (\cref{BoxicityConnections}) and graph colouring (\cref{ColouringConnections}).


\section{Layered Treewidth}
\label{LayeredTreewidth}

The starting point for the study of orthogonal tree decompositions is the notion of a layered tree decomposition, introduced independently by \citet{DMW17} and \citet{Shahrokhi13}. Applications of layered treewidth include nonrepetitive graph colouring  \citep{DMW17}, queue layouts, track layouts and 3-dimensional graph drawings \citep{DMW17}, book embeddings \citep{DF16}, and intersection graph theory \citep{Shahrokhi13}.

A \emph{layering} of a graph $G$ is a partition $(V_0,V_1,\dots,V_t)$ of $V(G)$ such that for every edge $vw\in E(G)$, if $v\in V_i$ and $w\in V_j$, then $|i-j|\leq 1$. Each set $V_i$ is called a \emph{layer}.  For example, for a vertex $r$ of a connected graph $G$, if $V_i$ is the set of vertices at distance $i$ from $r$, then $(V_0,V_1,\dots)$ is a layering of $G$. 

The \emph{layered width} of a tree decomposition $(T_x:x\in V(T))$ of a graph $G$ is the minimum integer $\ell$ such that, for some layering $(V_0,V_1,\dots,V_t)$ of $G$, each bag $T_x$ contains at most $\ell$ vertices in each layer $V_i$. The \emph{layered treewidth} of a graph $G$ is the minimum layered width of a tree decomposition of $G$. Note that the trivial layering with all vertices in one layer shows that layered treewidth is at most treewidth plus 1. The \emph{layered pathwidth} of a graph $G$ is the minimum layered width of a path decomposition of $G$; see \citep{BDDEW}.

While $n$-vertex planar graphs may have treewidth as large as $\sqrt{n}$, \citet{DMW17} proved the following\footnote{The \emph{Euler genus} of an orientable surface with $h$ handles is $2h$. The \emph{Euler genus} of  a non-orientable surface with $c$ cross-caps is $c$. The \emph{Euler genus} of a graph $G$ is the minimum Euler genus of a surface in which $G$ embeds (with no crossings).}:

\begin{thm}[\citep{DMW17}]
\label{DMW} 
Every planar graph has layered treewidth at most $3$. More generally, 
every graph with Euler genus $g$ has layered treewidth at most $2g+3$.
\end{thm}

Layered treewidth is related to local treewidth, which was first introduced by \citet{Eppstein-Algo00} under the guise of the `treewidth-diameter' property. A graph class $\mathcal{G}$ has \emph{bounded local treewidth} if there is a function $f$ such that for every graph $G$ in $\mathcal{G}$, for every vertex $v$ of $G$ and for every integer $r\geq0$, the subgraph of $G$ induced by the vertices at distance at most $r$ from $v$ has treewidth at most $f(r)$; see \citep{Grohe-Comb03,DH-SJDM04,DH-SODA04,Eppstein-Algo00}. If $f(r)$ is a linear function, then  $\mathcal{G}$ has \emph{linear local treewidth}. \citet{DMW17} observed that if every graph in some class $\mathcal{G}$ has layered treewidth at most $k$, then $\mathcal{G}$ has linear local treewidth with $f(r) \leq k(2r+1) -1$. \citet{DMW17} also proved the following converse result for minor-closed classes, where a graph $G$ is \emph{apex} if $G-v$ is planar for some vertex $v$. (Earlier, \citet{Eppstein-Algo00} proved that (b) and (d) are equivalent, and  \citet{DH-SODA04} proved that (b) and (c) are equivalent.)

\begin{thm}[\citep{DMW17,DH-SODA04,Eppstein-Algo00}]
\label{MinorClosedLayered}
The following are equivalent for a minor-closed class $\mathcal{G}$ of graphs:
\begin{enumerate}[label=(\alph*),itemsep=0ex,topsep=0ex]
\item $\mathcal{G}$ has bounded layered treewidth.
\item $\mathcal{G}$ has bounded local treewidth.
\item $\mathcal{G}$ has linear local treewidth.
\item $\mathcal{G}$ excludes some apex graph as a minor. 
\end{enumerate}
\end{thm}

\citet{DEW17} observed that such a converse result does not hold for non-minor-closed classes. In particular, 3-dimensional grid graphs have quadratic local treewidth and unbounded layered treewidth. 

A number of non-minor-closed classes also have bounded layered treewidth. \citet{DEW17} gave the following two examples. A graph is \emph{$(g,k)$-planar} if it can be drawn in a surface of Euler genus at most $g$ with at most $k$ crossings on each edge. \citet{DEW17} determined an optimal bound on the layered treewidth and treewidth of such graphs. 

\begin{thm}[\citep{DEW17}]
\label{gkPlanar}
Every $(g,k)$-planar graph $G$ has layered treewidth at most $(4g+6)(k+1)$ and treewidth at most $2\sqrt{(4g+6)(k+1)n}$. Conversely, for all $g,k\geq 0$ and infinitely many $n$ there is an $n$-vertex $(g,k)$-planar graph with treewidth 
$\Omega(\sqrt{(g+1)(k+1)n})$ and layered treewidth $\Omega((g+1)(k+1))$.
\end{thm}

Map graphs are defined as follows. Start with a graph $G_0$ embedded in a surface of Euler genus $g$, with each face labelled a `nation' or a `lake', where each vertex of $G_0$ is incident with at most $d$ nations. Define a graph $G$ whose vertices are the nations of $G_0$, where two vertices are adjacent in $G$ if the corresponding faces in $G_0$ share a vertex. Then $G$ is called a \emph{$(g,d)$-map graph}. A $(0,d)$-map graph is called a (plane) \emph{$d$-map graph}; such graphs have been extensively studied \citep{FLS-SODA12,Chen-JGT07,DFHT05,CGP02,Chen01}. It is easily seen that $(g,3)$-map graphs are precisely the graphs of Euler genus at most $g$ \citep{CGP02,DEW17}. So $(g,d)$-map graphs provide a natural generalisation of graphs embedded in a surface. Note that if a vertex of $G_0$ is incident with $d$ nations, then $G$ contains $K_d$, which need not be bounded by a function of $g$. \citet{DEW17} determined an optimal bound on the layered treewidth and treewidth of such graphs. 

\begin{thm}[\citep{DEW17}] 
\label{MapGraphs}
Every $(g,d)$-map graph on $n$ vertices has layered treewidth at most $(2g+3)(2d+1)$ and treewidth at most $2\sqrt{(2g+3)(2d+1)n}$. Moreover, for all $g\geq 0$ and $d\geq 8$, for infinitely many integers $n$, there is an  $n$-vertex $(g,d)$-map graph with treewidth at least $\Omega(\sqrt{(g+1)dn})$ and layered treewidth at least $\Omega((g+1)d)$.
\end{thm}

\cref{MinorClosedLayered} leads to further results. A tree decomposition is \emph{domino} if every vertex is in at most two bags \citep{BodEng-JAlg97,Bodlaender-DMTCS99,Wood09}. 

\begin{lem}
\label{BagsContainingVertex}
Every graph $G$ with layered treewidth $k$ has a domino path decomposition $P$ and a tree decomposition $T$ such that for every vertex $v$ of $G$, if $G_v$ is the subgraph of $G$ induced by the union of the bags of $P$ that contain $v$, then $T$ restricted to $G_v$ has width at most $3k-1$. 
\end{lem}

\preproof\begin{proof}
Let $T$ be a tree decomposition of $G$ with layered width $k$ with respect to some layering $V_1,\dots,V_t$ of $G$, where $V_t=\emptyset$. Then $P:=(V_1\cup V_2,V_2\cup V_3,\dots,V_{t-1}\cup V_t)$ is a path decomposition of $G$. Consider a vertex $v\in V_i$ for some $i\in[t-1]$. Then $v$ is in exactly two bags, $V_{i-1}\cup V_i$ and $V_i\cup V_{i+1}$. Thus $P$ is domino. The union of the bags that contain $v$ is  $V_{i-1}\cup V_i\cup V_{i+1}$, which contains at most $3k$ vertices in each bag of $T$.
\end{proof}

\cref{MinorClosedLayered,BagsContainingVertex} imply:

\begin{thm}
For every fixed apex graph $H$, there is a constant $k$, such that every $H$-minor-free graph $G$ has a domino path decomposition $P$ and a tree decomposition $T$ such that for every vertex $v$ of $G$, if $G_v$ is the subgraph of $G$ induced by the union of the bags of $P$ that contain $v$, then $T$ restricted to $G_v$ has width at most $3k-1$. 
\end{thm}

This result is best possible in the following sense. Let $G$ be obtained from the $n\times n$ grid graph by adding one dominant vertex $v$. Say $T_1$ and $T_2$ are tree decompositions of $G$. The bags of $T_1$ that contain $v$ induce a subgraph that contains the $n\times n$ grid, and therefore has treewidth at least $n$, which is unbounded. 

\section{Extremal Questions and Treewidth Bounds}
\label{Extremal}
\label{sqrtn}

We start this section by considering the natural extremal question: what is the maximum number of edges in an $n$-vertex graph that has two orthogonal tree decompositions of a particular type? \citet{DMW17} proved that every $n$-vertex graph with layered treewidth $k$ has minimum degree at most $3k-1$ and thus has at most $(3k-1)n$ edges, which is tight up to a lower order term. More general structures allow for quadratically many edges. For example, $K_{n,n}$ has two 2-orthogonal path decompositions, as shown in \cref{Introduction}. Note that each of these decompositions has quadratic magnitude. We now show that a limit on the magnitude of one decomposition leads to a linear bound on the number of edges, even for tree decompositions. 

%
%


\begin{lem}
\label{LinearSize}
Let $S$ and $T$ be $k$-orthogonal tree decompositions of a graph $G$, where $S$ 
has magnitude $s$. Then $|E(G)|\leq (k-1)s$. In particular, if $s \leq c|V(G)|$ then $|E(G)|\leq c (k-1) |V(G)|$. 
\end{lem}

\preproof\begin{proof}
Each edge of $G$ is in $G[S_x]$ for some $x\in V(S)$. Since $T$ restricted to $G[S_x]$ has treewidth at most $k-1$, it follows that $G[S_x]$ has less than $(k-1)|S_x|$ edges. Thus 
\begin{equation*}
|E(G)| \leq \sum_x|E(G[S_x])| \leq \sum_x (k-1)|S_x| = (k-1)s.\qedhere
\end{equation*}
\end{proof}

One application of layered treewidth is that it leads to $O(\sqrt{n})$ treewidth bounds.

\begin{thm}[Norine; see~\citep{DMW17}]
\label{Norine} For every $n$-vertex graph $G$ with layered treewidth $k$,
$$\tw(G) \leq 2\sqrt{kn}-1.$$
\end{thm}

As an example,  \cref{DMW,Norine} imply that graphs with bounded Euler genus have treewidth $O(\sqrt{n})$. \citet{DEW17} observed that a standard trick applied with \cref{Norine} implies:

\begin{thm}[\citep{DEW17}]
\label{Pathwidth} For every $n$-vertex graph $G$ with layered treewidth $k$,
$$\pw(G) \leq 11\sqrt{kn}-1.$$
\end{thm}

We now generalise these results to the setting of orthogonal decompositions. A \emph{weak path decomposition} of a graph $G$ is a sequence $P_1,\dots,P_t$ of sets of vertices of $G$ called bags, such that $P_1\cup\dots\cup P_t = V(G)$, for every vertex $v$ of $G$ the set of bags that contain $v$ forms a subsequence, and for every edge $vw$ of $G$, both $v$ and $w$ are in $P_i\cup P_{i+1}$ for some $i\in\{1,\dots,t\}$ (where $P_{t+1}$ means $\emptyset$). Note that  a path decomposition is a weak path decomposition in which the final condition is strengthened to say that both $v$ and $w$ are in $P_i$ for some $i\in\{1,\dots,t\}$. If $P_1,\dots,P_t$ is a weak path decomposition, then $P_1\cup P_2, P_2\cup P_3,\dots,P_{t-1}\cup P_t$ is a path decomposition with at most twice the width of $P_1,\dots,P_t$. In this sense, there is little difference between weak path decompositions and path decompositions. The \emph{magnitude} of a  weak  path decomposition $P_1,\dots,P_t$ is  $\sum_{i\in[n]}|P_i|$. 

Observe that a layering is a weak path decomposition in which each vertex is in exactly one bag. Thus weak path decompositions with linear magnitude generalise the notion of a layering. In a weak path decomposition, each bag $P_i$ separates $P_1\cup\dots\cup P_{i-1}$ and $P_{i+1}\cup\dots\cup P_t$; that is, there is no edge between  $P_1\cup\dots\cup P_{i-1}$ and $P_{i+1}\cup\dots\cup P_t$. This property is the key to the next lemma, which generalises \cref{Norine} to the setting of weak path decompositions. A tree decomposition $T$ and weak path decomposition $P_1,\dots,P_t$ of a graph $G$ are \emph{$c$-orthogonal} if $|T_x\cap P_i|\leq c$ for all $x\in V(T)$ and $i\in [t]$. 

\begin{lem}
\label{TreewidthUpperBound}
Suppose that $T$ is a tree decomposition and $P$ is a weak path decomposition of a graph $G$, where
$T$ and $P$ are $k$-orthogonal and $P$ has magnitude $s$. Then $$\tw(G)\leq 2 \sqrt{ ks}-1.$$
\end{lem}

\preproof\begin{proof}
Let $t:=\ceil{\sqrt{s/k}}$. Label the bags of $P$ in order by $1,\dots,t, 1,\dots,t,\dots$. 
Since the magnitude of $P$ is $s$, for some $i \in \{1,\dots, t \}$ the bags labelled $i$ have total size at most $s/t$. 
Let $G'$ be the subgraph of $G$ obtained by deleting the bags labelled $i$. 
Since each bag of $P$ separates the bags of $P$ before and after it, 
each connected component of $G'$ is contained within $t-1$ consecutive bags of $P$. 
Thus $G'$ has a tree decomposition with bags of size at most $(t-1)k$. 
Add all the vertices in bags of $P$ labelled $i$ to every bag of this tree decomposition of $G'$. 
We obtain a tree decomposition of $G$ with bag size at most $(t-1)k+s/t \leq 2\sqrt{ks}$.
Thus $\tw(G)\leq 2\sqrt{ks}-1$. 
\end{proof}

Several comments on \cref{TreewidthUpperBound} are in order.

First we show that \cref{TreewidthUpperBound} cannot be strengthened for two tree decompositions with bounded intersections. Let $G$ be a bipartite graph with bipartition $(A,B)$ and maximum degree $\Delta$. Let $S$ be the star decomposition of $G$ with root bag $A$ and a leaf bag $N[w]$ for each vertex $w\in B$. Symmetrically, let $T$ be the star decomposition of $G$ with root bag $B$ and a leaf bag $N[v]$ for each vertex $v\in A$.  Observe that $S$ and $T$ are $\Delta$-orthogonal and both have magnitude $|V(G)|+|E(G)|$. Now, apply this construction with $G$ a random cubic bipartite graph on $n$ vertices. We obtain two $3$-orthogonal tree decompositions of $G$ both with magnitude $\frac{5}{2}n$. But it is well known that $G$ has treewidth $\Omega(n)$; see \citep{GM09} for example. Thus  \cref{TreewidthUpperBound} does not hold for two tree decompositions with bounded intersections. 

\medskip
We now show that \cref{TreewidthUpperBound} proves that certain graph classes have bounded expansion. 
A graph class $\mathcal{C}$ has \emph{bounded expansion} if there exists a function $f$ such that for every graph $G\in\mathcal{C}$, for every subgraph $G'$ of $G$, and for all pairwise disjoint balls $B_1,\dots,B_s$ of radius at most $r$ in $G'$, the graph obtained from $G'$ by contracting each $B_i$ into a vertex has average degree at most $f(r)$. If $f(r)$ is a linear or polynomial function, then $\mathcal{C}$ has \emph{linear} or \emph{polynomial  expansion}, respectively. See \citep{Sparsity} for background on graph classes with bounded expansion. \citet{DMW17} proved that graphs with bounded layered treewidth have linear expansion. In particular, in a graph of layered treewidth $k$ contracting disjoint balls of radius $r$ gives a graph of layered treewidth at most $(4r+1)k$, and thus with average degree $O(rk)$. 
This result can be extended as follows. A class $\mathcal{G}$ of graphs is \emph{hereditary} if for every graph $G\in\mathcal{G}$ every induced subgraph of $G$ is in $\mathcal{G}$. \citet{DN16} proved that for a hereditary graph class $\mathcal{G}$, if every graph $G\in\mathcal{G}$ has a separator of order $O(|V(G)|^{1-\epsilon})$ for some fixed $\epsilon>0$, then $\mathcal{G}$ has polynomial expansion. \cref{TreewidthUpperBound,RS} then imply the following.

\begin{prop}
\label{PolyExp}
Let $\mathcal{G}_k$ be the class of graphs $G$ such that every subgraph $G'$ of $G$ has a path decomposition with magnitude at most $k|V(G')|$ and a tree decomposition that are $k$-orthogonal.
Then $\mathcal{G}_k$ has polynomial expansion. 
\end{prop}

%

We now show that \cref{PolyExp} cannot be extended to the setting of two tree decompositions with bounded intersections. Let $G$ be the 1-subdivision of $K_{n,n}$, which has $N= n^2+2n$ vertices. Say the bipartition classes of $K_{n,n}$ are $V:=\{v_1,\dots,v_n\}$ and $W:=\{w_1,\dots,w_n\}$. Let $x_{i,j}$ be the division vertex for edge $v_iw_j$. Let $S$ be the star decomposition of $G$ with root bag $\{v_1,\dots,v_n\}$, and for each $j\in[n]$ have a leaf bag $V\cup\{w_j,x_{1,j},\dots,x_{n,j}\}$. Similarly, let $T$ be the star decomposition of $G$ with root bag $W$, and for each $i\in[n]$ have a leaf bag $W\cup\{v_i,x_{i,1},\dots,x_{i,n}\}$. Then the intersection of a bag from $S$ and a bag from $T$ has size at most $3$. Each of $S$ and $T$ have magnitude $O(N)$. On the other hand, contracting the edges incident to each vertex $v_i$ gives $K_{n,n}$ which has unbounded average degree. Thus the class of 1-subdivisions of balanced complete bipartite graphs does not have bounded expansion, but every graph in the class has two tree decompositions with bounded intersections and linear magnitude. 





\medskip
Finally, we consider bounds on pathwidth. It is well known that hereditary graph classes with treewidth $O(n^\epsilon)$, for some fixed $\epsilon\in(0,1)$, have pathwidth $O(n^\epsilon)$; see \citep{Bodlaender-TCS98,DEW17} for example. In particular, \cref{TreewidthUpperBound} and  Lemma~6.1 of \citet{DEW17} imply the following.

\begin{lem}
\label{PathwidthUpperBound}
Let $\mathcal{G}$ be a hereditary class of graphs, such that every $n$-vertex graph $G$ in $\mathcal{G}$ has a tree decomposition $T$ and a weak path decomposition $P$, such that $T$ and $P$ are $k$-orthogonal
and $P$ has magnitude at most $cn$. Then for every $n$-vertex graph $G$ in $\mathcal{G}$, 
$$\pw(G)\leq 11 \sqrt{ ckn}-1.$$
\end{lem}


\section{Minor-Closed Classes}
\label{MinorClosed}

This section shows that graphs in a fixed minor-closed class have a tree decomposition and a linear-magnitude path decomposition with bounded intersections. The following graph minor structure theorem by Robertson and Seymour is at the heart of  graph minor theory. In a tree decomposition $T$ of a graph $G$, the \emph{torso} of a bag $T_x$ is the subgraph obtained from $G[T_x]$ by adding, for each edge $xy\in E(T)$, all edges $vw$ where $v,w\in T_x\cap T_y$. A graph $G$ is \emph{$(g,p,k,a)$-almost-embeddable} if there is set $A$ of at most $a$ vertices in $G$ such that $G-A$ can be embedded in a surface of Euler genus $g$ with at most $p$ vortices of width at most $k$. (See \citep{DMW17} for the definition of vortex, which will not be used in the present paper.)\  A graph is \emph{$k$-almost-embeddable} if it is  $(k,k,k,k)$-almost-embeddable. If $G_1$ and $G_2$ are disjoint graphs, where $\{v_1,\dots,v_k\}$ and $\{w_1,\dots,w_k\}$ are cliques of equal size respectively in $G_1$ and $G_2$, then a \emph{clique-sum} of $G_1$ and $G_2$ is a graph obtained from $G_1\cup G_2$ by identifying $v_i$ with $w_i$ for each $i\in\{1,\dots,k\}$, and possibly deleting some of the edges $v_iv_j$. 

\begin{thm}[\citep{RS-GraphMinorsXVI-JCTB03}]
\label{GMST}
For every fixed graph $H$ there are constants $g,p,k,a$ such that every $H$-minor-free graph is obtained by clique-sums of $(g,p,k,a)$-almost-embeddable graphs. Alternatively, every $H$-minor-free graph has a tree decomposition in which each torso is $(g,p,k,a)$-almost embeddable. 
\end{thm}

\citet{DMW17} introduced the following definition to handle clique sums. Say a graph $G$ is \emph{$\ell$-good} if for every clique $K$ of size at most $\ell$ in $G$ there is a tree decomposition of $G$ of layered width at most $\ell$ with respect to some layering of $G$ in which $K$ is the first layer. \citet{DMW17} proved the $\ell= (k+1)(2g+2p+3)$ case of the following result; the proof when $\ell> (k+1)(2g+2p+3)$ is identical. 

\begin{thm}[\citep{DMW17}]
\label{NoApexGood}
For every integer $\ell\geq (k+1)(2g+2p+3)$, every $(g,p,k,0)$-almost-embeddable graph $G$ is $\ell$-good.
\end{thm}

\citet{DMW17} actually proved a result stronger than \cref{NoApexGood} that allowed for apex vertices only adjacent to vertices in the vortices, but we will not need that.  \citet{DMW17} proved that for $\ell\geq k$, if $G$ is a $(\leq k)$-clique-sum of $\ell$-good graphs $G_1$ and $G_2$, then $G$ is $\ell$-good. \cref{GMSTinduction} below generalises this result  allowing for apex vertices. We first need the following definition. Define $\omega(g,p,k,a)$ to be the maximum size of a clique in a $(g,p,k,a)$-embeddable graph. \citet{JW13} proved that  $\omega(g,p,k,a)\in\Theta(a+(k+1)\sqrt{g+p})$. Define $$\ell(g,p,k):=\max\{\omega(g,p,k,0),(k+1)(2g+2p+3)\}.$$

\begin{lem}
\label{GMSTinduction}
Let $G$ be a graph that has a tree decomposition $T$, such that $T_\alpha\cap T_\beta$ is a clique of $G$ for each edge $\alpha\beta\in E(T)$, and $G[T_\alpha]$ is $(g,p,k,a)$-almost embeddable for each node $\alpha\in V(T)$. Then $G$ has a set of vertices $A$, such that $G-A$ is $\ell$-good, where $\ell:=\ell(g,p,k)$. Moreover, for every non-empty clique $K$  in $G$ there is a tree decomposition $T^*$ of $G$, such that:
\begin{itemize}[itemsep=0ex,topsep=0ex]
\item $T^*$ restricted to $G-A$ has layered width at most $\ell$ with respect to some layering $L$ of $G-A$ in which $K-A$ is the first layer, 
\item $T^*$ restricted to $A$ has width at most $a-1$.
\end{itemize}
\end{lem}

\preproof\begin{proof}
We proceed by induction on $|E(T)|$. In the base case with $|E(T)|=0$, $G$ is $(g,p,k,a)$-almost embeddable, implying $G$ contains a set $A$ of  at most $a$ vertices, such that $G-A$ is $(g,p,k,0)$-embeddable.  \cref{NoApexGood} implies that $G-A$ is $\ell$-good. Since $K-A$ is a clique in $G-A$ of size at most $\omega(g,p,k,0)\leq \ell$, there is a tree decomposition $T^*$ of $G-A$, such that $T^*$ has layered width at most $\ell$ with respect to some layering $L$ of $G-A$ in which $K-A$ is the first layer. Add $A$ to every bag of $T^*$. We obtain the desired tree decomposition of $G$, in which $T^*$ restricted to $A$ has width at most $a-1$ since $|A|\leq a$. 
This proves the base case. 

Now assume that $|E(T)|>0$. Let $xy$ be an edge of $T$. Let $Q:=T_x\cap T_y$. By assumption, $Q$ is a clique of $G$. Let $T^1$ and $T^2$ be the component subtrees of $T-xy$, where each node of $T^1$ and $T^2$ inherits its bag from the corresponding node of $T$.  For $i\in\{1,2\}$, let $G^i$ be the subgraph of $G$ induced by the union of the bags in $T^i$. Then $T^i$ is a tree decomposition of $G^i$, such that  $T_\alpha\cap T_\beta$ is a clique of $G^i$ for each edge $\alpha\beta\in E(T^i)$, and $G^i[T^i_\alpha]$ is $(g,p,k,a)$-almost embeddable for each node $\alpha\in V(T^i)$. Note that $G$ is obtained by pasting $G^1$ and $G^2$ on $Q$. 

Without loss of generality, the given clique $K$ of $G$ is in $G^1$. By induction, $G^1$ has a set of vertices $A^1$ and a tree decomposition $T^1$, such that $T^1$ restricted to $G^1-A^1$ has layered width at most $\ell$ with respect to some layering $L^1$ of $G^1-A^1$ in which $K-A^1$ is the first layer of $L^1$, and $T^1$ restricted to $A^1$ has width at most $a-1$. In $L^1$, the clique $Q-A^1$ is contained in one layer or in two consecutive layers. Let $Q'$ be the subclique of $Q-A^1$ contained in the first layer of $L^1$ that intersects $Q-A^1$. Note that if $(K\cap Q)\setminus A'\neq\emptyset$ then $Q'=(K\cap Q)\setminus A^1$.

By induction, $G^2$ has a set of vertices $A^2$ and a tree decomposition $T^2$, such that $T^2$ restricted to $G^2-A^2$ has layered width at most $\ell$ with respect to some layering $L^2$ of $G^2-A^2$ in which $Q'-A^2$ is the first layer of $L^2$, and $T^2$ restricted to $A^2$ has width at most $a-1$. Since $Q'\setminus A^2$ is the first layer, $(Q\setminus Q')\setminus A^2$ is contained within the second layer. 

Let $T^*$ be obtained from $T^1$ and $T^2$ by adding an edge between a bag of $T^1$ that contains $Q$ and a bag of $T^2$ that  contains $Q$. Since $Q$ is a clique, such bags exist. Now $T^*$ is a tree decomposition of $G$. Let $A:=A^1\cup A^2$. Then $T^*$ restricted to $A$ has width at most $a-1$, since $T^1$ restricted to $A^1$ has width at most $a-1$ and $T^2$ restricted to $A^2$ has width at most $a-1$. 

Delete each vertex in $A^2$ from each layer of $L^1$, and delete each vertex in $A^1$ from each layer of $L^2$. Now, no vertex in $A$ appears in a layer of $L^1$ or $L^2$. In particular, the first layer of $L^1$ equals $K\setminus A$. 

Construct a layering $L$ of $G-A$ by overlaying $L^1$ and $L^2$ so that the layer of $L^1$ that contains $Q'$ is merged with the first layer of $L^2$ (which equals $Q'-A^2$), and the layer of $L^1$ that contains $(Q\setminus Q')\setminus A^1$ is merged with the second layer of $L^2$ (which contains $(Q\setminus Q')\setminus A^2$). Then $L$ is a layering of $G-A$, since the vertices in common between $G^1-A$ and $G^2-A$ are exactly the vertices in $Q-A$. 

Consider the first layer of $L$, which consists of $K\setminus A$, plus any vertices added in the construction of $L$. If no such vertices are added, then the first layer of $L$ equals $K\setminus A$, as desired. Now assume that some vertices are added. Then the first layer of $L^1$ was merged with the first layer of $L^2$. Thus, by construction,  the first layer of $L^1$ contains $Q'$. Thus $Q'\subseteq K\setminus A^1$ and $Q'\setminus A^2\subseteq K\setminus A$. Thus the first layer of $L^2$ is a subset of the first layer of $L^1$, and the first layer of $L$ equals the first layer of $L^1$, which equals  $K\setminus A$, as desired. 

For each bag $T^*_\alpha$ of $T^*$ the intersection of $T^*_\alpha$ with a single layer of $L$ is a subset of the intersection of $T^*_\alpha$ and the corresponding layer in $L^1$ or $L^2$. Hence $T^*$ restricted to $G-A$ has layered width at most $\ell$ with respect to $L$.
%
\end{proof}

\cref{GMSTinduction} leads to the following theorem.

\begin{thm}
\label{ExcludedMinorTreePathStructure}
For every fixed graph $H$ there is a constant $k$, such that every $H$-minor-free graph $G$ has a tree decomposition $T^*$, a weak path decomposition $P$ of magnitude at most $k|V(G)|$, and a set of vertices $A$, such that 
\begin{itemize}[itemsep=0ex,topsep=0ex]
\item $T^*$ and $P$ are $k$-orthogonal, 
\item $T^*$ restricted to $A$ has width at most $k$, 
\item $P$ restricted to $G-A$ is a layering $L$, and
\item $T^*$ restricted to $G-A$ has layered width at most $k$ with respect to $L$. 
\end{itemize}
\end{thm}

\preproof\begin{proof}
\cref{GMST} says there are constants $g,p,k,a$ such that $G$ has a tree decomposition $T$ in which each torso is $(g,p,k,a)$-almost embeddable. Add edges to $G$ so that the intersection of any two adjacent bags in $T$ is a clique. 
Now, $G[T_\alpha]$ is $(g,p,k,a)$-almost embeddable for each node $\alpha\in V(T)$. By \cref{GMSTinduction}, $G$ has a set of vertices $A$ and a tree decomposition $T^*$, such that  $T^*$ restricted to $G-A$ has layered width at most $\ell=\ell(g,p,k,a)$ with respect to some layering $L$ of $G-A$, and $T^*$ restricted to $A$ has width at most $a-1$. 

For each node $\alpha\in V(T)$, add every vertex in $T_\alpha \cap A$ to every layer of $L$ that intersects $T_\alpha$. This produces a weak path decomposition $P$ of $G$, since in the proof of \cref{GMSTinduction}, if $A\cap Q\neq\emptyset$ then $A$ is added to the at most two layers containing $Q\setminus A$, implying that each vertex in $A$ is added to a consecutive subset of the layers. Moreover, $|T^*_x\cap P_y|\leq \ell+a$ for each node $x\in V(T^*)$ and node $y\in V(P)$, since  $T^*_x$ contains at most $\ell$ vertices in the layer corresponding to $y$, and at most $a$ vertices in $T^*_x\cap A$ are added to $P_y$. Finally, we bound the magnitude of $P$. Since each vertex of $G-A$ is in exactly one layer of $L$, the size of $L$ equals $|V(G-A)|$. For each node $\alpha\in V(T)$, the bag $T_\alpha$ uses at least two layers, and except for the first layer used by  $T_\alpha$, there is at least one vertex in $T_\alpha\setminus A$ in each layer used by $T_\alpha$. For each such layer, at most $a$ vertices in $T_\alpha\cap A$ are added to this layer in the construction of $P$. Thus there at most $2a |T_\alpha \setminus A|$ occurences of vertices in $T_\alpha\cap A$ in $P$. Thus the total number of occurrences in $P$ of vertices in $A$ is at most $2a|V(G-A)|$. Hence the magnitude of $P$ is at most $(2a+1)|V(G-A)|\leq (2a+1) |V(G)|$. 

The result follows with $k:=\max\{\ell+a,2a+1\}$. 
%
%
%
%
%
%
%
\end{proof}



\cref{RS,TreewidthUpperBound,PathwidthUpperBound,ExcludedMinorTreePathStructure} imply the following result due to \citet{AST90}, reproved by  \citet{Grohe-Comb03} and \citet{KawaReed10}. 
 It is interesting that our approach using orthogonal decompositions also reproves this result. 
\begin{thm}
\label{ExcludedMinorTreewidth}
For every fixed graph $H$, every $H$-minor-free graph on $n$ vertices has treewidth and pathwidth at most $O(\sqrt{n})$ and has a 
separator of order $O(\sqrt{n})$. 
\end{thm}


\section{String Graphs}
\label{StringGraphs}

A \emph{string graph} is the intersection graph of a set of curves in the plane with no three curves meeting at a single point \citep{PachToth-DCG02,SS-JCSS04,SSS-JCSS03,Krat-JCTB91,FP10,FP14}. For an integer $k\geq 2$, if each curve is in at most $k$ intersections with other curves, then the corresponding string graph is called a \emph{$k$-string graph}. Note that the maximum degree of a $k$-string graph might be much less than $k$, since two curves might have multiple intersections. A \emph{$(g,k)$-string} graph is defined analogously for curves on a surface of Euler genus at most $g$. 

\begin{thm}
\label{LayeredTreewidthString}
Every $(g,k)$-string graph has layered treewidth at most $2(k-1)(2g+3)$.
\end{thm}

\preproof\begin{proof}
Let $X$ be a set of curves in a surface of Euler genus at most $g$, such that no three curves meet at a point and each curve is in at most $k$ intersections with other curves in $X$. Let $G$ be the corresponding $(g,k)$-string graph. Let $G'$ be the graph obtained from $G$ by replacing each intersection point of two curves in $X$ by a vertex, where each curve is now a path on at most $k$ vertices. Thus $G'$ has Euler genus at most $g$. By \cref{DMW}, $G'$ has a tree decomposition $T'$ with layered width at most $2g+3$ with respect to some layering $V'_0,V_1',\dots,V_t'$. For each vertex $v$ of $G$, if $x_1,\dots,x_\ell$ is the path representing $v$ in $G'$, then $\ell\leq k$ and $x_1,\dots,x_\ell$ is contained in at most $k$ consecutive layers of $G'$. 

For each vertex $x$ of $G'$, let $T'_x$ be the subtree of $T$ formed by the bags that contain $x$. 
Let $T$ be the decomposition of $G$ obtained by replacing each occurrence of a vertex $x$ in a bag of $T'$ by the two vertices of $G$ that correspond to the two curves that intersect at $x$. 
We now show that $T$ is a tree decomposition of $G$. For each vertex $v$ of $G$, let $T_v$ be the subtree of $T$ formed by the bags that contain $v$. If $x_1,x_2,\dots,x_\ell$ is the path in $G'$ representing a vertex $v$ of $G$, then $T_v=T'_{x_1}\cup\dots\cup T'_{x_\ell}$, which is connected since each $T'_i$ is connected, and $T'_i$ and $T'_{i+1}$ have a node in common (containing $x_i$ and $x_{i+1}$). For each edge $vw$ of $G$, if $x$ is the vertex of $G'$ at the intersection of the curves representing $v$ and $w$, then $T_v$ and $T_w$ have $T'_x$ in common. Thus there is a bag containing both $v$ and $w$. Hence $T$ is a tree decomposition of $G$. 

For each vertex $v$ of $G$, let $f(v)$ be the minimum integer $i$ such that $V'_i$ contains a vertex $x$ of $G'$ in the curve corresponding to $v$. For $i\geq 0$, let $V_i:=\{v\in V(G): i(k-1) \leq f(v) \leq (i+1)(k-1)-1\}$. 
Then $V_0,V_1,\dots$ is a partition of $V(G)$. Consider an edge $vw$ of $G$ with $f(v) \leq f(w)$ and $v\in V_i$. Then the path in $G'$ representing $v$ is contained in layers $V'_{f(v)},V'_{f(v)+1},\dots,V'_{f(v)+k-1}$. Thus $f(w)\leq f(v)+k-1\leq (i+1)(k-1)-1+(k-1)\leq (i+2)(k-2)-1$. Since $f(w)\geq f(v)\geq i(k-1)$ we have $w\in V_i\cup V_{i+1}$. Hence $V_0,V_1.\dots$ is a layering of $G$.

Since each layer in $G$ is formed from at most $k-1$ layers in $G'$, and each layer in $G'$ contains at most $2g+3$ vertices in a single bag, each of which is replaced by two vertices in $G$, the layered treewidth of this decomposition is at most $2(2g+3)(k-1)$. 
\end{proof}


Every intersection graph of segments in the plane with maximum degree $k\geq 2$ is a $(0,k)$-string graph. Thus  \cref{Norine,Pathwidth,LayeredTreewidthString}  (since $(0,k)$-string graphs are a hereditary class) imply:

\begin{cor}
Every intersection graph of $n$ segments in the plane with maximum degree $k\geq 2$ has layered treewidth at most $6(k-1)$ and treewidth at most $2\sqrt{6(k-1)n}$ and pathwidth at most $11\sqrt{6(k-1)n}$.
\end{cor}

We now show that this corollary is asymptotically tight.

\begin{prop}
\label{SegmentIntersectiongraphLowerBound}
For $k\geq 6$ and for infinitely many values of $n$, there is a set of $n$ segments in the plane, whose intersection graph has maximum degree $k$, layered treewidth at least $\tfrac{1}{256} (k-5)$, and treewidth at least $\tfrac18 \sqrt{(k-5)n}-1$. 
\end{prop}

\preproof\begin{proof}
Let $G$ be a planar graph such that $\deg(v)+\deg(w)\leq k$ for every edge $vw$. 
By F\'ary's Theorem \citep{Fary48}, there is a crossing-free drawing of $G$ with each edge a segment. 
Then the intersection graph of $E(G)$ is the line graph of $G$, denoted by $L(G)$. 
Note that the degree of a vertex in $L(G)$ corresponding to an edge $vw$ in $G$ equals $\deg(v)+\deg(w)$. 

Let $r:=\floor{\frac{k-2}{4}}$. Infinitely many values of $n$ satisfy $n=4q^2r$ for some integer  $q\geq 1$. Let $Y'_{q,r}$ be the plane graph obtained from the $(q+1)\times(q+1)$ grid graph by subdividing each edge $r$ times, then adding a vertex of degree $4r$ inside each internal face adjacent to the subdivision vertices, and finally  deleting the grid edges and the non-subdivision vertices of the grid. Note that $\deg(v)+\deg(w)\leq 4r+2\leq k$ for each edge $vw$ of $Y'_{q,r}$. 

Observe that the line graph $L(Y'_{q,r})$ has $n$ vertices, and is exactly the graph $Z_{q,q,r}$ introduced by \citet{DEW17} and illustrated in the lower part of Figure~3 in \citep{DEW17}. By Lemma~5.6 in \cite{DEW17}, every separator of $L(Y'_{q,r})$ has size at least $\frac{qr}{2}=\tfrac14 \sqrt{rn} \geq \tfrac18 \sqrt{(k-5)n}$. 
By \cref{RS}, the treewidth of $L(Y'_{p,q,r})$ is at least $\tfrac18 \sqrt{(k-5)n}-1$. As shown above, $L(Y'_{p,q,r})$ is the intersection graph of $n$ segments in the plane, whose intersection graph has maximum degree $k$.

It follows from \cref{Norine} that $L(Y'_{p,q,r})$ has layered treewidth at least  $\tfrac{1}{256} (k-5)$. 
\end{proof}

The next result shows that \cref{LayeredTreewidthString} is asymptotically tight for all $g$ and $k$. The proof is omitted, since it is almost identical to the proof of Theorem~5.7 in \citep{DEW17} (which is equivalent to the second half of 
\cref{MapGraphs}).

\begin{prop} 
\label{MapGraphTreewidthLowerBound}
For all $g\geq 0$ and $k\geq 8$, for infinitely many integers $n$, there is an  $n$-vertex $(g,k)$-string graph with layered treewidth $\Omega(k(g+1))$ and treewidth $\Omega(\sqrt{(k(g+1)n})$.
\end{prop}


\cref{Norine,LayeredTreewidthString} imply that every $(g,k)$-string graph on $n$ vertices has treewidth at most $2\sqrt{2(k-1)(2g+3)n}$ and pathwidth $11\sqrt{2(k-1)(2g+3)n}$ (since $(g,k)$-string graphs are a hereditary class). However, this result is qualitatively weaker than the following theorem, which can be concluded from a recent result of \citet{DN14} and a separator theorem for string graphs by \citet{FP08}. See \citep{HW17} for a thorough discussion on the connections between separators and treewidth. 

\begin{thm}[\citep{DN14,FP08}] 
\label{TreewidthCrossings}
For every collection of curves on a surface of Euler genus $g$ with $m$ crossings in total, the corresponding string graph has a separator of order $O(\sqrt{(g+1)m})$ and treewidth $O(\sqrt{(g+1)m})$. 
\end{thm}

\citet{FP10} conjectured that \cref{TreewidthCrossings} can be improved in the $g=0$ case, with the assumption of ``$m$ crossings'' replaced by ``$m$ crossing pairs''. Equivalently, they conjectured that every string graph with $m$ edges has a $O(\sqrt{m})$ separator. \citet{FP10} proved that every string graph with $m$ edges has a $O(m^{3/4}\sqrt{\log m})$ separator; see \citep{FP14} for related results. The conjecture was almost proved by \citet{Mat14,Mat15}, who showed an upper bound of $O(\sqrt{m}\log m)$. Recently, \citet{Lee16} announced a proof of this $O(\sqrt{m})$ conjecture.


%

We now give an alternative proof of \cref{TreewidthCrossings} with explicit constants. The key is the following structure theorem of interest in its own right. The proof is analogous to that of \cref{LayeredTreewidthString}. 

\begin{lem}
\label{StringStructure}
For every collection of curves on a surface of Euler genus $g$ with $m$ crossings in total  (where no three curves meet at a point), the corresponding string graph has a tree decomposition $T$ and a  path decomposition $P$ such that $|T_x\cap P_y|\leq 2(2g+3)$ for all $x\in V(T)$ and $y\in V(P)$, and $P$ has magnitude $2m$. 
\end{lem}

\preproof\begin{proof}
Let $X$ be a collection of curves on a surface of Euler genus $g$ with $m$ crossings in total. Let $G$ be the corresponding string graph.  Let $G'$ be the graph obtained from $G$ by replacing each intersection point of two curves in $X$ by a vertex, where each curve crossed by $k$ other curves corresponds to a path on $k$ vertices in $G'$. Thus $G'$ has Euler genus at most $g$. By \cref{DMW}, $G'$ has a tree decomposition $T'$ with layered width at most $2g+3$ with respect to some layering $V'_0,V_1',\dots,V_t'$. 

For each vertex $x$ of $G'$, let $T'_x$ be the subtree of $T$ formed by the bags that contain $x$. 
Let $T$ be the decomposition of $G$ obtained by replacing each occurrence of a vertex $x$ in a bag of $T'$ by the two vertices of $G$ that correspond to the two curves that intersect at $x$. 
We now show that $T$ is a tree decomposition of $G$. For each vertex $v$ of $G$, let $T_v$ be the subtree of $T$ formed by the bags that contain $v$. If $x_1,x_2,\dots,x_k$ is the path in $G'$ representing a vertex $v$ of $G$, then $T_v=T'_{x_1}\cup\dots\cup T'_{x_k}$, which is connected since each $T'_i$ is connected, and $T'_i$ and $T'_{i+1}$ have a node in common (containing $x_i$ and $x_{i+1}$). For each edge $vw$ of $G$, if $x$ is the vertex of $G'$ at the intersection of the curves representing $v$ and $w$, then $T_v$ and $T_w$ have $T'_x$ in common. Thus there is a bag containing both $v$ and $w$. Hence $T$ is a tree decomposition of $G$. 

Construct a weak path decomposition $V_0,\dots,V_t$ as follows. For each vertex $x$ in $V'_i$ corresponding to the crossing point of two curves in $X$ corresponding to two vertices $v$ and $w$ in $G$, add $v$ and $w$ to $V_i$. For each vertex $v$ of $G$, if $(x_1,x_2,\dots,x_k)$ is the path in $G'$ representing $v$, then since $(x_1,x_2,\dots,x_k)$ is connected in $G'$, the set of bags that contain $v$ are consecutive. For each edge $vw$ of $G$, if $x$ is the crossing point between $v$ and $w$, then both $v$ and $w$ are in the bag $V_i$ where $x$ is in $V'_i$. Thus $V_0,\dots,V_t$ is a path decomposition of $G$.

Since each layer in $G'$ contains at most $2g+3$ vertices in a single bag, each of which is replaced by two vertices in $G$, we have 
$|T_x\cap P_y|\leq 2(2g+3)$ for all $x\in V(T)$ and $y\in V(P)$. Observe that $\sum_y |P_y| =2 |V(G')|=2m$. 
\end{proof}

Note that \cref{StringStructure} cannot be strengthened to say that string graphs with $O(n)$ crossings have bounded layered treewidth. For example, the graph obtained by adding a dominant vertex to the line graph of a $\sqrt{n}\times \sqrt{n}$ grid  is a string graph with $O(n)$ crossings, but the layered treewidth is $\Omega(n)$ (since the diameter is 2). This says that we need a path decomposition (rather than a layering) to conclude \cref{StringStructure}. \cref{RS,TreewidthUpperBound,StringStructure} imply:

\begin{thm}
\label{StringTreewidth}
For every collection of curves on a surface of Euler genus $g$ with $m$ crossings in total  (where no three curves meet at a point), the corresponding string graph has treewidth at most $4\sqrt{ (2g+3)m}-1$ and has a separator of order $4\sqrt{ (2g+3)m}$. 
\end{thm}

\section{Crossing Number}
\label{CrossingNumber}

Throughout this section we assume that in a drawing of a graph, no three edges cross at a single point. The \emph{crossing number} of a graph $G$ is the minimum number of crossings in a drawing of $G$ in the plane. See \citep{Schaefer14} for background on crossing numbers. This section shows that graphs with given crossing number have orthogonal decompositions with desirable properties. From this we conclude interesting lower bounds on the crossing number that in a certain sense, improve on known lower bounds. All the results generalise for drawings on arbitrary surfaces. 

\begin{thm}
\label{CrossingsConstruction}
Suppose that some $n$-vertex graph $G$ has a drawing on a surface of Euler genus $g$ with $m$ crossings in total. Then $G$ has a tree decomposition $T$ and a weak path decomposition $P$, such that $T$ and $P$ are $(4g+6)$-orthogonal
and $P$ has magnitude $2m+n$. 
\end{thm}

\preproof\begin{proof}
Orient each edge of $G$ arbitrarily. Let $G'$ be the graph obtained from $G$ by introducing a vertex at each crossing point. So $G'$ has $n+m$ vertices, and has Euler genus at most $g$. For a vertex $z$ of $G'-V(G)$ that corresponds to the crossing point of directed edges $v_1v_2$ and $w_1w_2$ in $G$, we say that $z$ \emph{belongs} to $v_1$ and $w_1$. Each vertex of $G$ \emph{belongs} to itself. 

By \cref{DMW}, $G'$ has  layered treewidth at most $2g+3$. That is, $G'$ has tree decomposition $T'$ and a layering $P'$ such that $|T'_x\cap P'_y| \leq 2g+3$ for each bag $T'_x$ and layer $P'_y$. For each vertex $z$ of $G'$ that belongs to $v_1$ and $w_1$ replace  each occurrence of $z$ in $T'$ and in $P'$ by both $v_1$ and $w_1$. Let $T$ and $P$ be the decompositions of $G$ obtained from $T'$ and $P'$ respectively. 

For each vertex $v$ of $G$ the set of vertices of $G'$ that belong to $v$ form a (connected) star centred at $v$. Thus the set of bags in $P$ that contain $v$ forms a (connected) subpath of $P$. Similarly, the set of bags in $T$ that contain $v$ forms a (connected) subtree of $T$. For each directed edge $vw$ of $G$, if $z$ is the last vertex in $G'$ before $w$ on the path from $v$ to $w$ corresponding to $vw$ (possibly $z=v$), then $zw\in E(G')$ and thus $z$ and $w$ are in a bag of $T'$, which implies that $v$ and $w$ are in a bag in $T$. Similarly, $z$ and $w$ are in a common bag of $P'$ or are in adjacent bags in $P'$, which impies that $v$ and $w$ are in a common bag of $P$ or are in adjacent bags in $P$. 

Hence $T$ is a tree decomposition and $P$ is a weak path decomposition of $G$, such that $|T_x\cap P_y|\leq 4g+6$ for $x\in V(T)$ and $y\in V(P)$. The total number of vertices in $P$ is $2|V(G')\setminus V(G)|+|V(G)|=2((n+m)-n)+n=2m+n$.
\end{proof}

\cref{TreewidthUpperBound,CrossingsConstruction} imply that if $G$ is a graph with a drawing on a surface of Euler genus $g$ with $m$ crossings in total, then 
\begin{align}
\label{twcr}
\tw(G) & \leq  2 \sqrt{ (4g+6)(2m+n)}-1
.
\end{align}
Let $\crossnum(G)$ be the crossing number of a graph $G$ (in the plane). 
Inequality \eqref{twcr} with $g=0$ can be rewritten as the following lower bound on $\crossnum(G)$:
\begin{align}
\label{CRtw} \crossnum(G) +\frac{1}{2}|V(G)| & \geq \frac{1}{48}(\tw(G)+1)^2 .
\end{align}
Of course, \eqref{CRtw}  generalises to the crossing number on any surface. We focus on the planar case since this is of most interest. 
Inequality \eqref{CRtw} is  similar to the following lower bounds on the crossing number in terms of bisection width $\bw(G)$ (due to \citet{PSS-Algo96} and \citet{SV94}) and cutwidth $\cw(G)$ and pathwidth $\pw(G)$ (due to \citet{DV-JGAA03}):
$$ \crossnum(G) +\frac{1}{16}\sum_{v\in V(G)}\deg(v)^2 \geq \frac{1}{40}\bw(G)^2 $$
$$ \crossnum(G) +\frac{1}{16}\sum_{v\in V(G)}\deg(v)^2 \geq \frac{1}{1176}\cw(G)^2 $$
$$\crossnum(G)+\sum_{v\in V(G)}\deg(v)^2 \geq \frac{1}{81} \pw(G)^2 .$$
In one sense, 
inequality \eqref{CRtw} is stronger than these lower bounds, since it replaces a $\sum_v\deg(v)^2$ term by a term linear in $|V(G)|$. On the other hand, $\bw(G)$ and $\cw(G)$ might be much larger than $\tw(G)$. For example, the star graph has treewidth and pathwidth 1, but has linear bisection width and linear cutwidth. And 
$\pw(G)$ might be much larger than $\tw(G)$. For example, the complete binary tree of height $h$ has treewidth 1 and pathwidth $\ceil{h/2}$. 

\section{Two Path Decompositions}
\label{TwoPaths}

This section considers graphs that have two path decompositions with bounded intersections. This property can be interpreted geometrically as follows:

\begin{obs}
A graph $G$ has two $k$-orthogonal path decompositions if and only if $G$ is a subgraph of an intersection graph of axis-aligned rectangles with maximum clique size at most $k$. 
\end{obs}

Of course, every bipartite graph is a subgraph of an intersection graph of axis-aligned lines with at most two lines at a single point. So every bipartite graph has two 2-orthogonal path decompositions.  This is essentially a restatement of the construction for $K_{n,n}$ in \cref{Introduction}. 

The following result of \citet{BDDEW} is relevant. 

\begin{thm}[\cite{BDDEW}] 
Every tree has layered pathwidth $1$ and every outerplanar graph has layered pathwidth at most 2. 
\end{thm}

This result implies that every tree has two 2-orthogonal path decompositions, and every outerplanar graph has two 4-orthogonal path decompositions. (The sequence of consecutive pairs of layers defines the second path decomposition, as described in \cref{Extremal}.)\  After trees and  outerplanar graphs the next simplest class of graphs to consider are series parallel graphs, which are the graphs with treewidth 2, or equivalently those containing no $K_4$ minor. Every outerplanar graph is series parallel. However, we now prove that series parallel graphs behave very differently compared to trees and outerplanar graphs. 

\begin{thm}
\label{SeriesParallel}
There is no constant $c$ such that every series parallel graph has two $c$-orthogonal path decompositions.
\end{thm}

The edge-maximal series parallel graphs are precisely the 2-trees, which are defined recursively as follows. $K_2$ is a 2-tree, and if $vw$ is an edge of a 2-tree $G$, then the graph obtained from $G$ by adding a new vertex adjacent only to $v$ and $w$ is also a 2-tree. To prove the above theorem, we show in \cref{path-times-path} below that for every integer $k$ there is a 2-tree graph $G$ such that every intersection graph of axis-aligned rectangles that contains $G$ as a subgraph also contains a $k$-clique.

Throughout this paper, the word \emph{rectangle} means open axis-aligned
rectangle: a subset $R$ of $\mathbb{R}^2$ of the form $(x_1,x_2)\times (y_1,y_2)$,
where $x_1<x_2$ and $y_1<y_2$.  The rectangle $R$ has four \emph{corners}
$(x_1,y_1)$, $(x_1,y_2)$, $(x_2,y_1)$ and $(x_2,y_2)$. And $R$ has four \emph{sides} (closed
vertical or horizontal line segments whose endpoints are corners) called 
the \emph{left}, \emph{top}, \emph{right} and \emph{bottom} sides of $R$ in the obvious way.
A \emph{rectangle intersection graph} is a graph whose vertices are
rectangles and the edge between two rectangles $u$ and $w$ is present
if and only $u\cap w\neq \emptyset$.  The boundary of $R$---the union
of its four sides---is denoted by $\partial R$.

We  make use of the fact that the set of rectangles in the plane is
a Helly family (of order 2) \cite[Chapter~11]{bollobas:combinatorics}:
\begin{obs}\label{helly}
   If $u$, $v$, and $w$ are rectangles that pairwise intersect, then
   $u\cap v\cap w\neq\emptyset$.
\end{obs}
\cref{helly} follows from Helly's Theorem for real intervals
and the observation that a rectangle is the Cartesian product of
two real intervals (see \cite[Page~83]{bollobas:combinatorics} or
\cite{bollobas.duchet:helly}).

Let $v$ and $w$ be two rectangles with $R=v\cap w\neq \emptyset$ and
such that $w$ does not contain any corner of $v$.  We say that $(v,w)$
is an \emph{h-pair} if the left or right side of $R$ is contained in
$\partial v$.  We say that $(v,w)$ is a \emph{v-pair} if the top or bottom
side of $R$ is contained in $\partial v$.  If $(v,w)$ is not an h-pair
or a v-pair, then we call it an \emph{o-pair}.  Note that, since $w$
does not contain a corner of $v$, $(v,w)$ is exactly one of a v-pair,
an h-pair or an o-pair. See \cref{hvo}.

\begin{figure}
   \begin{tabular}{ccc}
   \includegraphics{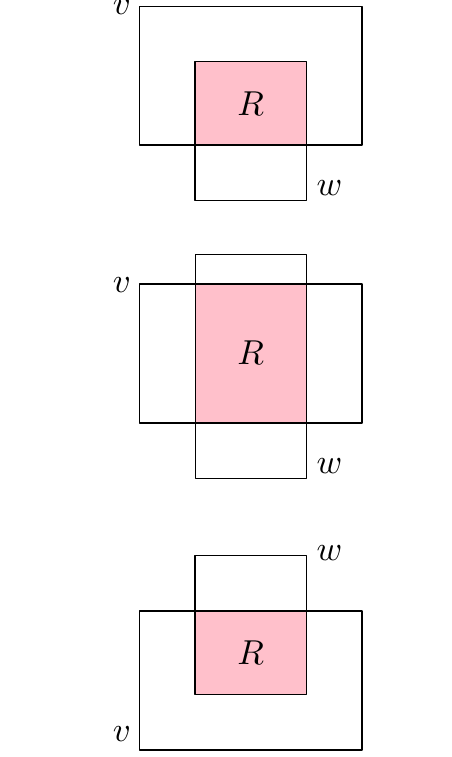} & 
   \includegraphics{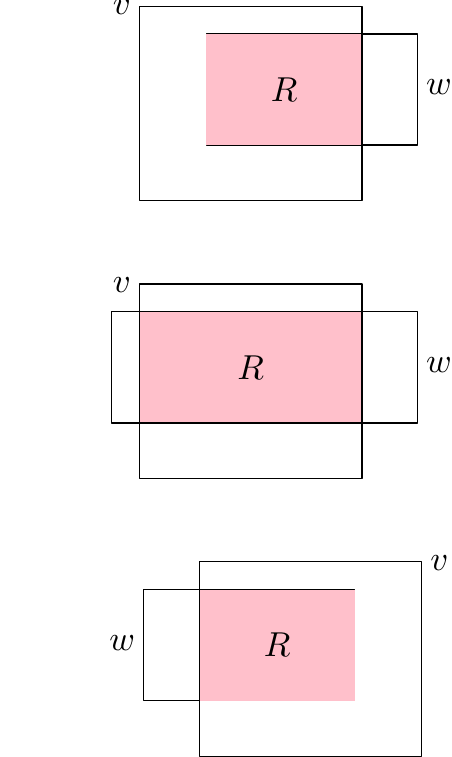} & 
   \includegraphics{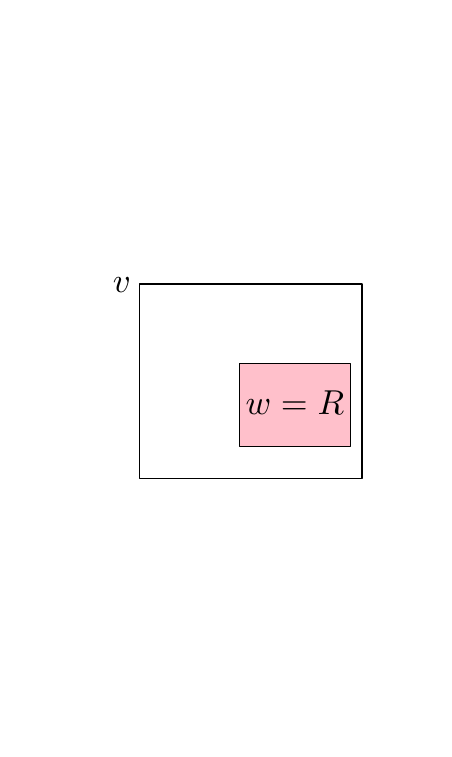} \\
   v-pairs & h-pairs & o-pair
   \end{tabular}
   \caption{\label{hvo}Examples of v-pairs, h-pairs, and an o-pair.}
\end{figure}

Our proof works by finding a path in a rectangle intersection graph $G$
that defines a sequence of rectangles having properties that ensure
that these rectangles form a clique. 

See \cref{good-sequence} for an illustration of the following definition:
Let $v_1,\ldots,v_k$ be a sequence of rectangles and let
$R_i=\bigcap_{j=1}^i v_j$.  We say that $v_1,\ldots,v_k$ is
\emph{hvo-alternating} if
\begin{enumerate}[topsep=0pt,itemsep=0pt]
  \item for each $i\in\{2,\ldots,k\}$, $v_i\cap R_{i-1}\neq \emptyset$;
  \item for each $i\in\{2,\ldots,k\}$,
        $v_i$ does not contain any corner of $R_{i-1}$; and
  \item for each $i\in\{2,\ldots,k-1\}$, $(R_{i-1},v_i)$ and
        $(R_i,v_{i+1})$ are not both h-pairs and not both v-pairs.
\end{enumerate}

\begin{figure}
  \begin{center}
    \includegraphics{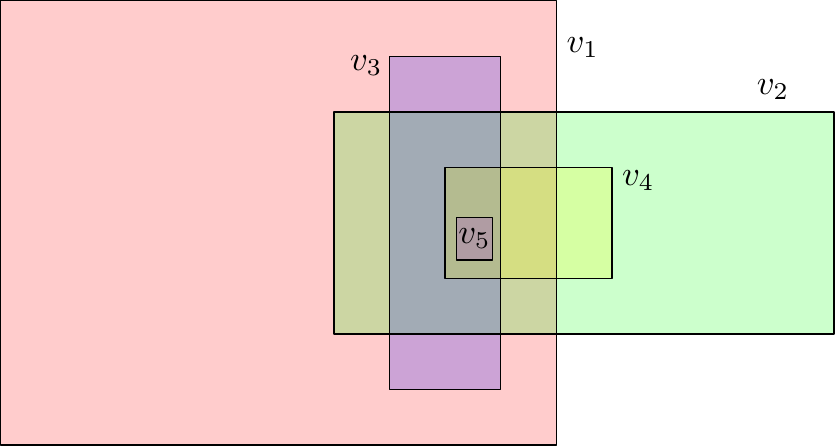}
  \end{center}
  \caption{\label{good-sequence}An hvo-alternating sequence of rectangles.}
\end{figure}

Note that Property~1 with $i=k$ ensures that $\bigcap_{j=1}^k v_i\neq\emptyset$.  Therefore, if $v_1,\ldots,v_k$ are
vertices in a rectangle intersection graph $G$, then these vertices form a
$k$-clique in $G$.  Our proof attempts to grow an hvo-alternating sequence
of vertices in $G$. The following lemma shows that an hvo-alternating
sequence is neatly summarized by its last two elements:

\begin{lem}
\label{hvo-alternating}
  If $v_1,\ldots,v_k$ is an hvo-alternating sequence of
  rectangles then $\bigcap_{i=1}^{k} v_i = v_{k-1}\cap v_k$.
\end{lem}

\preproof\begin{proof}
  The case $k=2$ is trivial, so we first consider the case $k=3$.
  Recall that, for any two sets $A$ and $B$, $A\supseteq B$ if and only
  if $A\cap B = B$. Therefore, it is sufficient to show that $v_1\supseteq
  v_2\cap v_3$.

  If $(v_1,v_2)$ is an o-pair, then $v_1\supseteq v_2\supseteq v_2\cap
  v_3$ and we are done.  Otherwise, assume without loss of generality
  that $(R_1,v_2)=(v_1,v_2)$ is an h-pair, so that $(R_2,v_3)$ is not an
  h-pair. Refer to \cref{geometric}.  In this case, $v_2\cap v_3$ is
  contained in the rectangle $R^*$ whose top and bottom sides coincide
  with those of $v_2$ and whose left and right sides coincide with
  those of $v_1$. We finish by Observing that $v_1\supseteq R^*\supseteq
  v_2\cap v_3$, as required.

  \begin{figure}
    \begin{center}
      \includegraphics{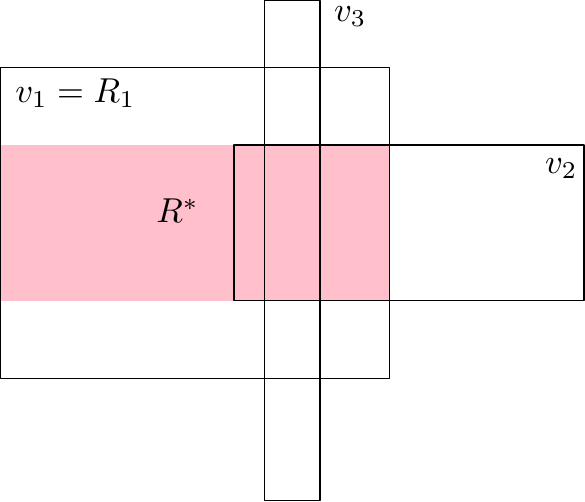}
    \end{center}
    \caption{\label{geometric}The proof of \cref{hvo-alternating}.}
  \end{figure}

  Next consider the general case $k> 3$.
  By \cref{hvo-alternating}, the three element sequence
  $R_{k-2},v_{k-1},v_k$ is an hvo-alternating sequence so, applying 
  the result for $k=3$, we obtain
  \[  
   \bigcap_{j=1}^k v_j = R_{k-2}\cap v_{k-1}\cap v_k = v_{k-1}\cap v_k
   \enspace .  \qedhere 
  \]
\end{proof}

Sometimes the process of growing our hvo-alternating sequence stalls. The
following lemma shows that, when this process stalls, we can at least
replace the last element in the sequence.

\begin{lem}
\label{replacement}
  Let $v_1,\ldots,v_k$ be an hvo-alternating sequence of rectangles and
  define $R_i=\bigcap_{j=1}^i v_i$. Let $v$ be a rectangle such
  that 
  \begin{enumerate}[itemsep=0ex,topsep=0ex] 
     \item $v\cap R_k\ne \emptyset$; 
     \item $v$ contains no corner of $R_k$; and 
     \item $v_1,\ldots,v_k,v$ is not hvo-alternating.
  \end{enumerate}
  Then $v_1,\ldots,v_{k-1},v$ is hvo-alternating.
\end{lem}

\preproof\begin{proof}
  Notice that $v_1,\ldots,v_k,v$ satisfies all the conditions to be
  hvo-alternating except that $(R_{k-1},v_k)$ and $(R_k,v)$
  are both h-pairs or both v-pairs.  Without loss of generality assume
  that they are both h-pairs.

  It is sufficient to show that $(R_{k-1},v)$ is not a v-pair so that,
  by replacing $v_k$ with $v$ we are replacing the h-pair $(R_{k-1},v_k)$
  with an h-pair or an o-pair.  But this is immediate, since
  $v$ intersects $R_k$ but does not intersect the top or bottom side
  of $R_k$.  Therefore $v$ cannot intersect the top or bottom side of
  $R_{k-1}\supseteq R_k$, so $(R_{k-1},v)$ is not a v-pair.
\end{proof}

If our process repeatedly stalls, then the hvo-alternating sequence we
are growing never gets any longer; we only repeatedly change the last
element in the sequence. Next we describe the sequences of rectangles
that appear during these repeated stalls and show that such sequences
(if long enough) also determine large cliques in $G$.

A sequence $v_1,\ldots,v_k$ of rectangles is \emph{h-nesting} with
respect to a rectangle $R$ if, for each $i\in\{1,\ldots,k\}$, $(R,v_i)$
is an h-pair and, for each $i\in\{2,\ldots,k\}$, $(R\cap v_{i-1},v_i)$
is an h-pair; see \cref{nesting}.  A \emph{v-nesting sequence} is
defined similarly by replacing h-pair with v-pair.

\begin{figure}
 \begin{center}
    \includegraphics{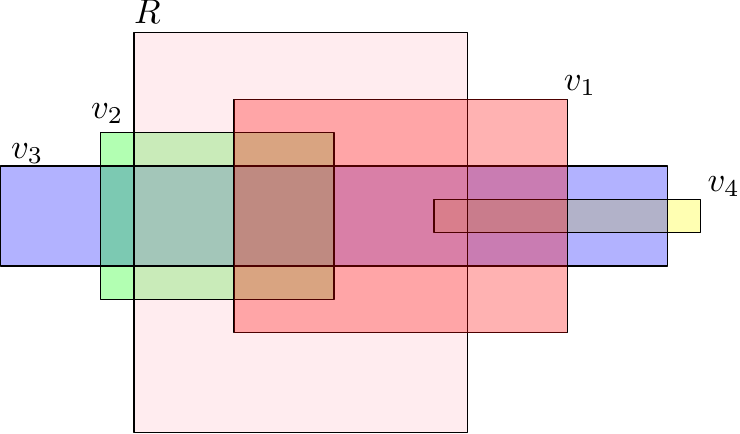}
 \end{center}
 \caption{ \label{nesting}Rectangles $v_1,\ldots,v_4$ that are h-nesting with respect to $R$.}
\end{figure}

\begin{lem}
\label{NestingLemma}
   If $v_1,\ldots,v_k$ is an h-nesting sequence or a v-nesting sequence
   with respect to $R$, then there exists a point $x\in R$ such that 
   $|\{ i: x\in v_i \}|\geq \lceil k/2\rceil$.
\end{lem}

\preproof\begin{proof}
   Assume, without loss of generality, that $v_1,\ldots,v_k$
   is an h-nesting sequence with respect to $R$.  Consider the
   sequence of horizontal strips $s_1,\ldots,s_k$, where each
   $s_i=(-\infty,\infty)\times (y_{i,1},y_{i,2})$ has top and
   bottom sides that coincide with those of $v_i$. Since, for each
   $i\in\{2,\ldots,k\}$, $(R,v_{i-1})$ and $(R\cap v_{i-1},v_i)$
   are both h-pairs, $s_1\supseteq s_2\supseteq\cdots\supseteq s_k$.
   Let $\ell$ be a point on the left side of $R$ contained in $s_k$
   and let $r$ be a point on the right side of $R$ contained in $s_k$.
   Since $(R,v_i)$ is an h-pair, $v_i$ contains at least one of $\ell$
   or $r$, for each $i\in\{1,\ldots,k\}$.  Therefore, at least one of
   $\ell$ or $r$ is contained in at least $\lceil k/2\rceil$ rectangles
   in $v_1,\ldots,v_k$.  Since rectangles are open, there is a point
   $x$, in the neighbourhood of $\ell$ or $r$ (as appropriate) that
   is contained in $R$ and is also contained in $\lceil k/2\rceil$
   rectangles in $v_1,\ldots,v_k$.
\end{proof}

We now introduce a particular 2-tree. 
The \emph{height-$h$ $d$-branching universal 2-tree}, $T_{h,d}$ is
defined recursively as follows:
\begin{itemize}[itemsep=0ex,topsep=0ex]
  \item $T_{-1,d}$ is the empty graph;
  \item $T_{0,d}$ is a two-vertex graph with a single edge;
  \item For $h\geq 1$, $T_{h,d}$ is obtained from $T_{h-1,d}$ by adding,
     for each edge $vw \in E(T_{h-1,d})\setminus E(T_{h-2,d})$, $d$
     new vertices $v_{1,vw},\ldots,v_{d,vw}$ each adjacent to both $v$
     and $w$.
\end{itemize}
The \emph{root edge} of $T_{h,d}$ is the single edge of $T_{0,d}$.
For each $i\in\{0,\ldots,h\}$, the level-$i$ vertices of $T_{h,d}$
are the vertices in $V(T_{i,d})\setminus V(T_{i-1,d})$.  The level-$i$
edges of $T$ are the edges that join a level-$i$ vertex to a level-$j$
vertex for some $j<i$.

Note that the number of level-$i$ edges in $T_{h,d}$ is given by the recurrence
\[
   m_i = \begin{cases}
           1 & \text{if $i=0$} \\
           2dm_{i-1} & \text{otherwise}
       \end{cases}
\]
which resolves to $(2d)^{i}$.  Thus, the total number of edges in
$T_{h,d}$ is less than $(2d)^{h+1}$.

We are now ready to prove the main result of this section.

\begin{thm}
\label{path-times-path}
  For each $k\in \N$, every rectangle intersection graph that contains
  $T_{4k-7,\iters}$ as a subgraph contains a clique of size $k$.
\end{thm}

\preproof\begin{proof}
  The cases $k=1$ and $k=2$ are trivial so, for the remainder of the proof
  we assume that $k\geq 3$.

  Let $G$ be a rectangle intersection graph that contains
  $T=T_{4k-7,\iters}$ as a spanning subgraph.  We use the convention that
  $V(G)=V(T)$ so that vertices of $T$ are rectangles in $V(G)$.  
  We will attempt to define a path $v_1,\ldots,v_k$ in $T$ such that
  $v_1,\ldots,v_k$ is hvo-alternating. 

  Let $r_0r_1$ be the root edge of $T$. We set $v_1=r_0$ and use the convention
  that $v_0=r_1$.  We will perform $\iters$ iterations, each of which
  tries to add another vertex, $v_{i+1}$, to a partially constructed
  hvo-alternating path $v_1,\ldots,v_i$. During iteration $t$, for
  each $t\in\{1,\ldots,\iters\}$, the procedure will consider a
  level-$t$ vertex, $v$, of $T$ to include in the path.  At the end of
  iteration $t$, the last vertex in the partially constructed sequence
  is always $v$, a level-$t$ vertex.

  At the beginning of iteration $t$, $v_i$ is a level-$(t-1)$ vertex, so
  $v_{i-1}$ and $v_i$ are both adjacent to a set $S$ of $4k-7$ level-$t$
  vertices.  Each of the rectangles in $S$ intersects $v_{i-1}$ and
  $v_i$ so, by \cref{helly}, each rectangle in $S$ intersects
  $v_{i-1}\cap v_i$.  Therefore, by \cref{hvo-alternating}, each of
  the rectangles in $S$ intersects $R_i=\bigcap_{j=1}^i v_j$.

  If each rectangle in $S$ contains at least one corner of $R_i$, then
  some corner of $R_i$ is contained in at least
  \[
         \left\lceil(4k-7)/4\right\rceil = k-1
  \]
  rectangles.  Since these rectangles are open, there is a point $x$
  contained in these $k-1$ rectangles and in $v_{i}$.  The resulting set
  of $k$ rectangles therefore form a $k$-clique in $G$ and we are done.

  Otherwise, some rectangle $v\in S$ does not contain a corner of $R_i$.
  Notice that the sequence $v_1,\ldots,v_i,v$ is hvo-alternating except,
  possibly, that the pairs $(R_{i-1},v_i)$ and $(R_i,v)$ are both h-pairs
  or both v-pairs.  There are two cases to consider:
  \begin{enumerate}[itemsep=0ex,topsep=0ex]
     \item $v_1,\ldots,v_i,v$ is hvo-alternating. In this
       case we say that the procedure \emph{succeeds} in iteration $t$
       and we set $v_{i+1}=v$.

     \item $v_1,\ldots,v_i,v$ is not hvo-alternating.  In this case,
       we say that the procedure \emph{stalls} in iteration $t$.
       In this case, we \emph{change} $v_i$ by setting $v_i=v$. By
       \cref{replacement} the resulting sequence is hvo-alternating,

       Note that in this case we have failed to make our path any
       longer. Instead, we have only replaced the last element with a
       level-$t$ vertex.  Regardless, the next iteration will try to
       extend the path with the new value of $v_i$.
  \end{enumerate}
  If we allow this procedure to run sufficiently long, then one of two
  cases occurs:
  \begin{enumerate}[itemsep=0ex,topsep=0ex]
     \item At least $k-1$ iterations are successes.  In this case, we
     find $v_2,\ldots,v_k$, so that $v_1,\ldots,v_k$ is an hvo-sequence
     whose vertices form a $k$-clique in $G$.

     \item Some element of our sequence, $v_i$ takes on a sequence
     $S_i=v_{i,0},\ldots,v_{i,2(k-i)}$ of $2(k-i)+1$ different values
     because the procedure stalls $2(k-i)$ times while trying to
     select $v_{i+1}$.  In this case, $S_i$ is either h-nesting or
     v-nesting with respect to $R_{i-1}$ so, by \cref{NestingLemma}, some
     subset $\{w_0,\ldots,w_{k-i}\}\subset S_i$ of rectangles in $S_i$
     all have a common intersection that includes a point of $R_{i-1}$.
     But $R_{i-1}$ is the common intersection of $v_1,\ldots,v_{i-1}$.
     Therefore $v_1,\ldots,v_{i-1},w_{0},\ldots,w_{k-i}$ all have a
     point in common and form a $k$-clique in $G$.
  \end{enumerate}
  The number of iterations required before this procedure finds
  a $k$-clique in $G$ is at most
  \[
      h = \sum_{i=2}^k (2(k-i)+1) = k + (k-2)(k+1) = (k-1)^2 \enspace .
  \]
  During these $h$ iterations, the procedure selects a level-$j$ vertex
  of $T$ for $j=1,2,\ldots,h$.  Since $T$ has height $(k-1)^2 = h$, the
  procedure succeeds in finding a $k$-clique before running out of levels.
\end{proof}

This completes the proof of \cref{SeriesParallel}. 

We now give an application of \cref{SeriesParallel} in the setting of graph partitions.

\begin{prop}
There is no constant $c$ such that every series-parallel graph has a vertex-partition into two induced subgraphs each with pathwidth at most $c$. 
\end{prop}

\preproof\begin{proof}
Suppose that every series-parallel graph $G$ has such a partition $\{V_1,V_2\}$ of $V(G)$, where $P_1$ and $P_2$ are path decompositions of $G[V_1]$ and $G[V_2]$ respectively, each with width at most $c$. Adding $V_2$ to every bag of $P_1$ and $V_1$ to every bag of $P_2$ gives two $(2c+2)$-orthogonal path decompositions of $G$, contradicting \cref{SeriesParallel}.
\end{proof}

This result is in contrast to a theorem of \citet{DDOSRSV04}, which says that for every fixed graph $H$ there is a constant $c$, such that every $H$-minor-free graph has a vertex-partition into two induced subgraphs each with treewidth at most $c$. 

\section{Boxicity Connections}
\label{BoxicityConnections}

\cref{TwoPaths} attempts to understand which graphs have two path decompositions with bounded intersections.  This turns out to be equivalent to the more geometric problem of understanding which graphs are subgraphs of rectangle intersection graphs of bounded clique size.  There are several other ways we could generalize these problems.


Does adding extra dimensions help?  A \emph{$d$-dimensional box intersection graph} is a
graph whose vertices are $d$-dimensional axis-aligned boxes and for which
two vertices are adjacent if and only if they intersect.  For a graph,
$G$, the smallest $d$ such that $G$ is a $d$-dimensional box intersection graph
is known as the \emph{boxicity} of $G$.

\begin{op}
   For each $d\in\N$, what is the smallest value of $r=r_d$ for which
   the following statement is true: For every $k\in\N$ there exists an $r$-tree
   $T$ such that every graph $G$ of boxicity $d$ that contains $T$
   as a subgraph contains a $k$-clique?
\end{op}

\cref{path-times-path} shows that $r_2 \leq 2$, and the obvious
representation of trees (1-trees) as rectangle intersection graphs shows
that $r_2> 1$, so $r_2=2$.
For $d>2$, we now show  that $r_d \leq 2d$.

\begin{prop}
\label{2dtree}
For every $k\geq 2$ and $d\geq 1$, there exists a $2d$-tree $T_k$ such that
every intersection graph of $d$-dimensional boxes that contains $T_k$ as a
subgraph contains a $k$-clique.
\end{prop}

\begin{proof}
The $2d$-tree $T_k$ is defined inductively: $T_0$ is a $(2d+1)$-clique. To
obtain $T_i$, attach a vertex adjacent to each $2d$-clique in $T_{i-1}$.

Now consider some box representation of a graph $G$ that contains $T_r$ as
a subgraph, so that the vertices of $G$ are $d$-dimensional boxes. 
We find the desired clique inductively by constructing two sequences of
sets $S_0,\dots,S_r$ and $X_0 \subseteq \dots \subseteq X_r$ 
that satisfy the following conditions for each $i\in \{0,1,\dots, r\}$:
\begin{enumerate}[itemsep=0pt,topsep=0pt]
\item $|S_i| = 2d$,
\item $S_i \subseteq X_i \subseteq V(T_i)$,
\item $S_i$ forms a clique in $T_i$,
\item $\cap X_i = \cap S_i$ (so $X_i$ is a $(2d+1+i)$-clique in $G$).
\end{enumerate}

Let $X_0 := V(T_0)$. To find $S_0$, look at the common intersection $\cap X_0$.  At least one of the boxes $x$ in $X_0$ is redundant in the sense that 
$\cap X_0 = \cap (X_0 \setminus {x})$. Let $S_0 := X_0 \setminus \{x\}$.  Observe that $X_0$ and
$S_0$ satisfy Conditions 1--4.

Now for the inductive step. By Conditions 1 and 3, the vertices of $S_i$ form a $2d$-clique in $T_i$, so
there is some vertex $v$ in $T_{i+1}$ that is adjacent to all the vertices
in $S_i$. Let $X_{i+1}:=X_i \cup \{v\}$ and let $S_{i+1}$ be the subset of $S_i \cup
\{v\}$ obtained by removing one redundant box.  Then $X_{i+1}$ and
$S_{i+1}$ satisfy Conditions 1--4.

Set $r:=\max\{0,k-2d-1\}$. Then every graph of boxicity at most $d$ that contains $T_r$ contains a $k$-clique.
\end{proof}

On the lower-bound side, a result of \citet{thomassen:interval}
states that every planar graph has a boxicity at most 3. Since every
2-tree is planar, this implies $r_d \geq 3$ for $d\geq 3$.  More generally,
\citet{chandran.sivadasan:boxicity} show that every graph $G$ has boxicity at most $\tw(G)+2$, so
every $(d-2)$-tree has boxicity at most $d$. This implies that $r_d\geq d-1$.

Returning to two dimensions, recall that the Koebe--Andreev--Thurston
theorem states that every planar graph is a subgraph of some intersection
graph of disks for which no point in the plane is contained in more
than two disks.  2-trees are a very special subclass of planar graphs,
so \cref{path-times-path} shows that axis-aligned rectangles are very
different than disks in this respect. Thus, we might ask how expressive
other classes of plane shapes are.


For a set $C$ of lines, we can consider intersection graphs
of \emph{$C$-oriented convex shapes}: convex bodies whose boundaries
consist of linear pieces, each of which is parallel to some line in $C$.
(Axis-aligned rectangles are $C$-oriented where $C$ is the set consisting
of the x-axis and y-axis.)  For a set $C$ of lines, let $\mathcal{G}_C$
denote the class of intersection graphs of $C$-oriented convex shapes.
For an integer $c$, let $\mathcal{G}_c=\bigcup_{C:|C|=c}\mathcal{G}_C$.

\begin{op}
   For each $c\in\N$, what is the smallest value of $r=r_c$ such that the
   following statement is true: For every $k\in\N$, there exists an
   $r$-tree $T$ such that every graph $G\in\mathcal{G}_c$ that contains
   $T$ as a subgraph contains a $k$-clique?
\end{op}

As in the proof of \cref{2dtree}, we can use the fact that, for any set $B$ of $C$-oriented convex shapes, there is a subset $B'\subset B$ with $|B'|\leq 2|C|$ and $\cap B=\cap B'$ to establish that $r_c \leq 2c$. We do not know a lower bound better than $r_c > 1$.

\section{Colouring Connections}
\label{ColouringConnections}

The \emph{tree-chromatic number} of a graph $G$, denoted by $\text{tree-}\chi(G)$, is the minimum integer $k$ such that $G$ has a tree decomposition in which each bag induces a $k$-colourable subgraph. This definition was introduced by \citet{Seymour16}; also see \citep{HK17,BFMMSTT17}. For a graph $G$, define the \emph{2-dimensional treewidth} of $G$, denoted $\ttw(G)$, to be the minimum integer $k$ such that $G$ has two $k$-orthogonal tree decompositions $S$ and $T$. For each bag $B$ of $S$, note that $T$ defines a tree decomposition of $G[B]$ with width $k-1$, implying $G[B]$ is $k$-colourable. Thus 
\begin{equation}
\label{treechi}
\ttw(G) \geq \text{tree-}\chi(G) \geq \omega(G).
\end{equation}
Obviously, $\ttw(G)\geq\omega(G)$ and $\text{tree-}\chi(G)\geq\omega(G)$. 
\Cref{treechi} leads to lower bounds on $\ttw(G)$. For example, \citet{Seymour16} showed that tree-chromatic number is unbounded on triangle-free graphs. Thus $\ttw$ is also unbounded on triangle-free graphs. 

On the other hand, we now show that there are graphs with bounded tree-chromatic number (even path-chromatic number) and unbounded $\ttw$. The \emph{shift graph} $H_n$ has vertex set $\{(i,j):1\leq i<j\leq n\}$ and edge set $\{(i,j)(j,\ell):1\leq i<j<\ell\leq n\}$. 
It is well known that $H_n$ is triangle-free with chromatic number at least $\log_2 n$; see \citep{HN04}. 
\citet{Seymour16} constructed a path decomposition of $H_n$ such that each bag induces a bipartite subgraph. We prove the following result that separates $\text{tree-}\chi$ and $\ttw$. 

\begin{thm}\label{SplitTheorem}
For every $k$ there exists $n$ such that the shift graph $H_n$ does not admit two $k$-orthogonal tree decompositions. 
\end{thm}

\begin{lem}\label{SplitLemma}
For every integer $\ell\geq 1$ there exists an even integer $m>\ell$ such that  
for every tree decomposition $T$ of the shift graph $H=H_m$, 
	\begin{description}
	\item[(i)] $\chi(H[T_x])>\ell$ for some $x\in V(T)$, or
	\item[(ii)] there exists $x \in V(T)$, $A  \subseteq \{1,2,\ldots,\frac{m}{2}\}$, and 
	$B  \subseteq \{ \frac{m}{2} +1,\ldots,m\}$, such that $|A|,|B| > \ell$, and  $(a,b) \in T_x$ for all $a \in A$ and $b \in B$.
	\end{description}
\end{lem}

\preproof\begin{proof}[Proof of \cref{SplitTheorem} assuming \cref{SplitLemma}.]  Let $\ell=m(k)$, and let $n=m(\ell)$, where $m$ is as in  \cref{SplitLemma}. Then $n>\ell>k$. We show that $n$ satisfies the theorem. Suppose not, and let $T$  and $T'$ be two $k$-orthogonal tree decompositions of $H_n$. Each bag $T_x$ induces a subgraph of treewidth  at most $k-1$ in $H_n$. Thus $\chi(H[T_x]) \leq k \leq \ell$. Thus  outcome (i) of \cref{SplitLemma} applied to $\ell$ and $H$ does not hold for $T$, and so outcome (ii) holds. Let $x$, $A$ and $B$  be as in this outcome. We may assume that $|A|=|B|$. Let $H'$ be the subgraph of $H$ induced by $\{(a,b)\in V(H): a,b \in A\cup B\}$. Then $H'\cong H_{|A\cup B|}$. Apply \cref{SplitLemma} to $H'$ and the tree decomposition  $T''$ of $H'$ induced by $T'$. Each bag of $T''$ induces a $k$-colourable subgraph of $H'$. Since $k<\ell$, outcome (ii) occurs. Thus there exist $A' \subseteq A$ and $B' \subseteq A$ and $x' \in V(T')$ such that $|A'|,|B'| \geq  k+1$ and   $(a,b) \in T'_{x'}$ for all $a \in A'$ and $b \in B'$. Thus $|T_x \cap T'_{x'}| \geq (k+1)^2 > k$, yielding the desired contradiction.
\end{proof}

\preproof\begin{proof}[Proof of \cref{SplitLemma}.] Define $m:=2^{4\ell+1}+2$. Let $T$ be a tree decomposition of the shift graph $H=H_m$. By splitting vertices of $T$, if necessary, we may assume that $T$ has maximum degree at most 3.   Let $H^1$ be the subgraph of $H$ induced by pairs $(a,b)$ such that $a,b \leq \frac{m}{2}$, and let  $H^2$ be the subgraph  induced by pairs $(a,b)$ with $a,b > \frac{m}{2}$. Then $H^1\cong H^2\cong H_{m/2}$. For a subtree $X$ of $T$ and $i \in \{1,2\}$, let $H^i(X)$ denote the subgraph of $H^i$ induced by $V(H^i) - \bigcup\{T_x:x\in V(T)\setminus V(X)\}$; that is, by those vertices of $H^i$ that only belong to the bags of $T$ corresponding to vertices of $X$. We may assume that (i) does not hold. 

Suppose that there does not exist an edge $e \in E(T)$ such that $\chi(H^1(X)) \geq \ell+1$ and $\chi(H^1(Y)) \geq \ell+1$, where $X$ and $Y$ are the components of $T-e$. That is, $\chi(H^1(X)) \leq \ell$ or $\chi(H^1(Y)) \leq \ell$. Orient $e$ towards $X$ if $\chi(H^1(X)) \leq \ell$ and towards $Y$ otherwise. Do this for every edge of $T$. Let $v$ be a source node in this orientation of $T$. Then $\chi(H^1(Z)) \leq \ell$ for each subtree $Z$ of $T$ with $v \not \in V(Z)$. Let $Z_1$, $Z_2$ and $Z_3$ be the (at most) three maximal subtrees of $T - v$. Then $H^1(Z_1)$, $H^1(Z_2)$, $H^1(Z_3)$ and $H^1[T_v]$ are four induced subgraphs of $H^1$ covering the vertex set of $H^1$, each with chromatic number at most $\ell$. Thus $\log_2 \frac{m}{2} \leq \chi(H^1) \leq 4\ell$, contradicting the choice of $m$. 

Thus there exists an edge $e \in E(T)$ such that $\chi(H^1(X)) \geq \ell+1$ and $\chi(H^1(Y)) \geq \ell+1$, where $X$ and $Y$ are the components of $T-e$. Repeating the argument in the previous paragraph for $H^2$, we may assume without loss of generality that $\chi(H^2(Y)) \geq \ell+1$. Let $A$ be the set of all $a $ such that $(b,a) \in V(H^1(X))$  for some $b < a$. Then $|A| \geq \ell+1$, since $H^1(X)$ can be properly coloured using one colour for each element of $A$. Similarly, if  $B$ is the set of all $b$ such that $(b,a) \in V(H^2(Y))$  for some $a > b$, then $|B| \geq \ell+1$. Consider $(a,b) \in V(H)$ with $a \in A$ and $b \in B$. Then $(a,b)$ has a neighbour in $V(H^1(X))$ and therefore $(a,b) \in T_x$ for some $x \in V(X)$. Similarly,   $(a,b) \in T_y$ for some $y \in V(Y)$. If $z$ is an endpoint of $e$, then $z$ lies on the $xy$-path in $T$, and thus $(a,b) \in T_z$ and (ii) holds.
\end{proof}

The above results suggest a positive answer to the following question. 
 
\begin{op}
\label{ttwChiBounded}
Is there a function $f$ such that every graph that has two $k$-orthogonal tree decompositions is $f(k)$-colourable? That is, is $\chi(G)\leq f(\ttw(G))$ for every graph $G$?
\end{op}

This question asks whether graphs that have two tree decompositions with bounded intersections have bounded chromatic number. Note that graphs that have two path decompositions with bounded intersections have bounded chromatic number (since rectangle intersection graphs are $\chi$-bounded, as proved by \citet{AG60}). Also note that the converse to this question is false: there are 3-colourable graph classes with $\ttw$ unbounded. Suppose the complete tripartite graph $K_{n,n,n}$ has two $k$-orthogonal tree decompositions. Then each decomposition has at least $2n+1$ vertices in some bag (since $\tw(K_{n,n,n})=2n$). The intersection of these two bags has size at least $n+2$. Thus $k\geq n+2$ and $\ttw(K_{n,n,n})\geq n+2$.

\paragraph{Note.} Independently of this paper, \citet{Stav15,Stav16} introduced the \emph{$i$-medianwidth} of a graph for each $i\geq 1$. It is easily seen that $\ttw(G)$ equals the 2-medianwidth of $G$.  Following the initial release of this paper, \citet{FJMTW17} answered  \cref{ttwChiBounded} in the negative. 



\bibliographystyle{myNatbibStyle}
\bibliography{myBibliography}
\end{document}